\documentclass[12pt]{amsart}
\usepackage{amssymb,amsmath,amsthm,slashed}
\newtheorem{theorem}{Theorem}
\newtheorem{lemma}[theorem]{Lemma}
\newtheorem{corollary}[theorem]{Corollary}

\newtheorem{definition}[theorem]{Definition}

\theoremstyle{remark}

\numberwithin{theorem}{section} \numberwithin{equation}{section}

\newcommand{\im}{\textnormal{Im}}

\newcommand{\beq}{\begin{eqnarray*}}
\newcommand{\eeq}{\end{eqnarray*}}

\def\CP1{\mathbb{C}\mathrm{P}^1}

\title{Kummer surfaces associated with Seiberg-Witten curves}

\author{Andreas Malmendier}

\address{Department of Mathematics, Colby College, 
Waterville, ME 04901}
\email{andreas.malmendier@colby.edu}

\begin{document}
\begin{abstract}
By carrying out a rational transformation on the base curve $\mathbb{CP}^1$ of the Seiberg-Witten curve
for $\mathcal{N}=2$ supersymmetric pure $\mathrm{SU}(2)$-gauge theory, we obtain a family of Jacobian elliptic 
K3 surfaces of Picard rank $17$. The isogeny relating the Seiberg-Witten curve for pure $\mathrm{SU}(2)$-gauge theory
to the one for $\mathrm{SU}(2)$-gauge theory with $N_f=2$ massless hypermultiplets extends to define a Nikulin
involution on each K3 surface in the family. We show that the desingularization of the quotient of the K3 surface
by the involution is isomorphic to a Kummer surface of the Jacobian variety of a curve of genus two.
We then derive a relation between the Yukawa coupling associated with the elliptic K3 surface and the
Yukawa coupling of pure $\mathrm{SU}(2)$-gauge theory.
\end{abstract}

\maketitle

\section{Introduction}

\noindent
In physics, families of K3 surfaces are a crucial ingredient in 
the compactification of string theory.  From the viewpoint of string theory, 
elliptically fibered K3 surfaces whose fibers are complex one-dimensional tori $E_{\tau }=\mathbb{C}/(\mathbb{Z}\oplus \mathbb{Z}\tau )$ 
where $\tau \in \mathbb{H}$ and $\mathbb{H}$ is the complex upper half-plane define F-theory vacua in eight dimensions where the scalar field $%
\tau $ of the type-IIB string theory is allowed to be multi-valued \cite{MorrisonVafa}. 

In collaboration with David Morrison \cite{MalmendierMorrison}, we previously
constructed two families of Jacobian elliptic K3 surfaces of Picard rank $17$. 
For each K3 surface $\mathbf{X}$, we found a Shioda-Inose structure, i.e., an
automorphism $\imath $ of order two preserving the holomorphic two-form, such
that $\mathbf{X}/\imath $ is birational to the Kummer surface of the Jacobian 
variety of a curve of genus two. We interpreted the families 
of elliptic K3 surfaces of Picard rank $17$ to describe the moduli spaces of
the two heterotic string theories with large gauge group compactified on a two-torus $%
T^{2}$ in the presence of one Wilson line.

The limit when the Jacobian degenerates to a product of two elliptic curves $
F_{1}\times F_{2}$ describes a well-understood case of the
F-theory/heterotic string duality in the absence of Wilson lines. 
The moduli space of elliptic K3 surfaces with $\mathrm{H}\oplus 
\mathrm{E}_{8}\oplus \mathrm{E}_{8}$ polarization can be identified with the
moduli space of the heterotic string vacua with gauge group $(\mathrm{E}%
_{8}\times \mathrm{E}_{8})\rtimes \mathbb{Z}_{2}$ and $\mathrm{Spin}(32)/%
\mathbb{Z}_{2}$ respectively compactified on $T^{2}$ \cite{MorrisonVafa}. 
The K\"{a}hler metric and $B$-field identify $T^2$ with the elliptic
curve $F_{1}$ and $F_{2}$ respectively. 

In this paper, we give an explicit construction of 
a family of Jacobian elliptic K3 surfaces with high Picard rank
that stems from a different approach. Donaldson theory on the one hand 
and F-theory on the other, come together through an idea provided by Sen \cite{Sen}: 
the identification of the complex gauge coupling in Donaldson theory with the axion-dilaton modulus in
string theory provides an embedding of gauge theory into F-theory. 
To this end, we will prove the following result:
\begin{theorem}
\label{maintheorem}
Let $p(t) = 4 \, t^3 - 3  \,  a  \, t - b$ and $r(t) =t - c$ with $a, b, c\in \mathbb{C}$. We also set
$p_\pm(t) = p(t) \pm r^2(t)$ and denote the roots of $p_+(t)=0$ and $p_-(t)=0$ by $t_2, t_4, t_6$
and $t_1, t_3, t_5$ respectively. We have that
\begin{enumerate}
\item[]
\item The family of Jacobian elliptic surfaces $\mathbf{X}_{a,b,c}$ over $\mathbb{CP}^1$ defined by
\begin{equation*}
\mathbf{X}_{a,b,c} = \left\lbrace (\hat{Y},\hat{X},t) \in \mathbb{C}^3 \mid \hat{Y}^2 = 4 \, \hat{X}^3 - \hat{g}_2 \, \hat{X} - \hat{g}_3 \right\rbrace \;,
\end{equation*}
where $t$ is the affine coordinate on the base curve $\mathbb{CP}^1$ and
\begin{eqnarray*}
\hat{g}_2     & = & \; \frac{1}{3} \, p^2(t) - \frac{1}{4}\, r^4(t) \;,\\
\hat{g}_3     & = & \frac{1}{216} \, p(t) \, \Big( 8 \, p^2(t) - 9 \, r^4(t)  \Big) \;,
\end{eqnarray*}
is a family of $ \langle 8 \rangle  \oplus \mathrm{E}_8  \oplus \mathrm{E}_8$-lattice polarized K3 surfaces.

\item The elliptic K3 surface $\mathbf{X}_{a,b,c} \to \mathbb{CP}^1$ admits a Van Geemen-Sarti involution.

\item For the rational surjective map 
$f_{a,b,c}: \mathbb{CP}^1 \twoheadrightarrow \mathbb{CP}^1$ given by $t \mapsto u = p(t)/r^2(t)$, the elliptic surface $\mathbf{X}_{a,b,c}$
is obtained from the Seiberg-Witten curve for pure $\mathrm{SU}(2)$-gauge theory, viewed as a Jacobian elliptic fibration
$\mathbf{X}_{\scriptscriptstyle \mathrm{SW}} \to \mathbb{CP}^1$  with analytic marking $\hat{\omega}_{\scriptscriptstyle \mathrm{{SW}}}$, by taking the fibered product via $f_{a,b,c}$, i.e., 
\begin{equation*}
\begin{array}{rclcc}
 \mathbf{X}_{a,b,c}  =  \mathbf{X}_{\scriptscriptstyle \mathrm{{SW}}} & \times_{\mathbb{CP}_u^1}   & \mathbb{CP}_t^1 & \longrightarrow             & \mathbb{CP}_t^1  \\
                                                   & \downarrow                 &                 &                             & \phantom{f_{a,b,c}}\downarrow f_{a,b,c}\\
                                                   & \mathbf{X}_{\scriptscriptstyle \mathrm{{SW}}} &                 & \longrightarrow             & \mathbb{CP}_u^1  
\end{array} 
\end{equation*}
The base change replaces the smooth fiber $\hat{E}_u$ of $\mathbf{X}_{\scriptscriptstyle \mathrm{{SW}}}$ by smooth fibers $\hat{E}_t$ with the same $\tau$-parameter, i.e.,
\begin{equation*}
 \tau(\hat{E}_u) = \tau(\hat{E}_t) \;\; \text{for $u=f_{a,b,c}(t)$}\;.
\end{equation*}
The holomorphic cubic forms $\hat{\Xi}_{\scriptscriptstyle \mathrm{{SW}}}$ on $\mathbf{X}_{\scriptscriptstyle \mathrm{SW}}$ and $\hat{\Xi}$ on $\mathbf{X}_{a,b,c}$ are related by
\begin{equation*}
 f_{a,b,c}^* \Big(\hat{\Xi}_{\scriptscriptstyle \mathrm{SW}}\Big) = r^2(t) \, \left( \dfrac{\partial u}{\partial t} \right)^2 \;  \hat{\Xi} \;.
\end{equation*}
  \item For generic values of the parameters $a, b, c$,  the K3 surface $\mathbf{X}_{a,b,c}$ admits a rational map of degree two
   onto the Kummer surface of the Jacobian variety of a generic curve $\mathbf{C}$ of genus two.
           Equations~(\ref{parameter_abc}), (\ref{parameter_ABC}), and (\ref{parameter}) express the complex 
          parameters $a, b, c$ defining $\mathbf{X}_{a,b,c}$ in terms of the roots of the sextic curve (\ref{curve_2}) defining $\mathbf{C}$.
\end{enumerate}
\end{theorem}

\noindent 
Throughout this article we call a choice of values for the complex parameters $a, b, c$ generic if the Jacobian elliptic fibration $\mathbf{X}_{a,b,c}$ has the singular fibers
$ I_8 \oplus 6 \, I_1 \oplus I_4^*$. 

\newpage
\noindent
Other explicit examples of Jacobian elliptic K3 surfaces of Picard rank 17 that are the Kummer surfaces
of Jacobian varieties of generic curves of genus two have been constructed previously by Hoyt \cite{Hoyt89,Hoyt},
Kumar \cite{Kumar}, and Clingher-Doran \cite{CD3,CD4}.

The outline of this paper is as follows. In Section \ref{prelim}, we recall facts about K3 surfaces, in
particular Kummer surfaces associated with curves of genus two. In Section \ref{K3families}, we define 
two 3-parameter families $\mathbf{X}_{a,b,c}$ and $\mathbf{Y}_{a,b,c}$ of Jacobian elliptic K3 surfaces of Picard rank $17$. 
We also show that each surface $\mathbf{X}_{a,b,c}$ admits a Van Geemen-Sarti involution and a rational map of degree two onto $\mathbf{Y}_{a,b,c}$.
In Section \ref{SeibergWitten_curves}, we consider the rational elliptic surfaces $\mathbf{X}_{\scriptscriptstyle \mathrm{{SW}}}$ and $\mathbf{Y}_{\scriptscriptstyle \mathrm{{SW}}}$
which are the Seiberg-Witten curves for $\mathcal{N}=2$ supersymmetric pure $\mathrm{SU}(2)$-gauge theory
and $\mathrm{SU}(2)$-gauge theory with $N_f=2$ massless hypermultiplets respectively.
By carrying out a rational transformation on the base curve $\mathbb{CP}^1$ of the Seiberg-Witten curves $\mathbf{X}_{\scriptscriptstyle \mathrm{{SW}}}$ and $\mathbf{Y}_{\scriptscriptstyle \mathrm{{SW}}}$
we obtain the families of Jacobian elliptic K3 surfaces $\mathbf{X}_{a,b,c}$ and $\mathbf{Y}_{a,b,c}$, respectively.
The two-isogeny relating the Seiberg-Witten curve for $N_f=0$ to the one for $N_f=2$ extends to define a Nikulin involution on $\mathbf{X}_{a,b,c}$ 
and a degree-two rational map onto $\mathbf{Y}_{a,b,c}$.  In Section~\ref{MonodromyPeriodsYukawa}, we derive a simple relation
between the monodromies around the singular fibers, the period integrals, and the Yukawa couplings 
for the elliptic K3 surfaces and the Seiberg-Witten curves.
In Section~ \ref{kummer_surface}, we show that for generic values of $a, b, c$ the K3 surface $\mathbf{Y}_{a,b,c}$ is
the Kummer surfaces for the Jacobian variety of a generic curve $\mathbf{C}$ of genus two and express the parameters $a, b, c$ in terms 
of the roots of the sextic curve defining $\mathbf{C}$.
In Section \ref{proof}, we combine these results to prove Theorem \ref{maintheorem}.

\section{Preliminaries}\label{prelim}
\noindent
Recall that a K3 surface is a smooth simply connected complex projective surface with trivial canonical bundle.
If $\mathbf{Y}$ is a K3 surface it follows $H^2(\mathbf{Y},\mathbb{Z}) \cong  \mathrm{H}^3 \oplus \mathrm{E}_8(-1)^2$. 
Here, $\mathrm{H}$ is the lattice $\mathbb{Z}^2$ with the quadratic form $2 xy$,
and $\mathrm{E}_8(-1)$ is the negative definite lattice
associated with the exceptional root systems of $\mathrm{E}_8$.

For a complex two-dimensional torus $\mathbf{Z}$ it follows $H^2(\mathbf{Z},\mathbb{Z}) \cong \mathrm{H}^3$. 
The algebraic complex tori of dimension one are elliptic curves. But most complex tori of dimension two are 
not algebraic. The algebraic ones are called Abelian surfaces. Examples are products of two elliptic curves and the 
Jacobian varieties of smooth projective curves of genus two. Let $\mathbf{Z}$ be a two-dimensional Abelian variety. The map $-\mathbb{I}$ has sixteen distinct fixed points whence
$\mathbf{Z} / \{ \pm \mathbb{I} \}$ is a singular surface with $16$ rational double points. The minimal resolution of $\mathbf{Z} / \{ \pm \mathbb{I} \}$
is a special K3 surface called the Kummer surface $\mathrm{Kum}(\mathbf{Z})$.

Let $\mathbf{Y}$ be a Abelian or K3 surface. 
The class $\omega \in H^2(\mathbf{Y},\mathbb{C})$ of the non-vanishing holomorphic 
two-form is unique up to multiplication by a scalar. The polarized Hodge structure of weight two will be denoted 
as follows:
\begin{equation*}
 \begin{array}{ccccccc}
  H^2(\mathbf{Y},\mathbb{C}) & = & H^{2,0}(\mathbf{Y}) & \oplus & H^{1,1}(\mathbf{Y}) & \oplus & 
                                         H^{0,2}(\mathbf{Y}) \\
 && \parallel && \parallel && \parallel \\
 && \langle \omega \rangle_{\mathbb{C}} && \langle \omega, \overline{\omega} \rangle_{\mathbb{C}}^{\perp}  &&
  \langle \overline{\omega} \rangle_{\mathbb{C}} 
 \end{array}
\end{equation*}
A polarization is given by the intersection form, i.e., a non degenerate integer symmetric bilinear 
form on $H(\mathbf{Y},\mathbb{Z})$ extended to $H^2(\mathbf{Y},\mathbb{C})$ by linearity.
A principally polarized Abelian surface is either the Jacobi variety of a smooth projective curve of genus two where the polarization 
is the class of the theta divisor,  or is the product of two elliptic curves with the product polarization.

The Picard group $\mathrm{Pic}(\mathbf{Y})$ is the group of Cartier divisors modulo linear equivalence and can be identified
with $H^1(\mathbf{Y}, \mathcal{O}^*_{\mathbf{Y}})$. The kernel of the first Chern class $c_1: \mathrm{Pic}(\mathbf{Y}) \to H^2(\mathbf{Y})$ 
is denoted by $\mathrm{Pic}^0(\mathbf{Y})$ and the quotient $\mathrm{Pic}(\mathbf{Y})/\mathrm{Pic}^0(\mathbf{Y}) = \mathrm{NS}(\mathbf{Y})$
is the N\'eron-Severi group. The Picard group together with the intersection form is an Euclidean lattice.
The Picard number $\rho(\mathbf{Y})$ is the $\mathrm{rank}$ of  $\mathrm{NS}(\mathbf{Y})$, and the
N\'eron-Severi lattice is an even lattice of signature $(1,\rho(\mathbf{Y})-1)$.

The first Chern class restricts to an isomorphism $c_1: \mathrm{Pic}(\mathbf{Y}) \to H^2(\mathbf{Y},\mathbb{Z}) \cap H^{1,1}(\mathbf{Y})$
by the Lefschetz theorem. 
$H^1(\mathbf{Y}, \mathcal{O}_{\mathbf{Y}})$ maps onto $\mathrm{Pic}^0(\mathbf{Y})$. If $\mathbf{Y}$ is an elliptic surface over $\mathbb{CP}^1$ 
it follows that $H^1(\mathbf{Y}, \mathcal{O}_{\mathbf{Y}})=0$. Hence, for an elliptic K3 surface the natural map $\mathrm{Pic}(\mathbf{Y}) \to \mathrm{NS}(\mathbf{Y})$ 
is an isomorphism. The orthogonal complement $\mathrm{T}_{\mathbf{Y}} = \mathrm{NS}(\mathbf{Y})^{\perp} \in H^2(\mathbf{Y},\mathbb{Z})\cap H^{1,1}(\mathbf{Y})$ is
called the transcendental lattice and carries the induced Hodge structure. A Hodge isometry between two transcendental
lattices of Abelian or K3 surfaces is an isometry preserving the Hodge structure. 

A lattice polarization on the algebraic K3 surface is given by a primitive lattice embedding
$\mathrm{N} \hookrightarrow \mathrm{NS}$ whose image contains a pseudo-ample class. 
Here $\mathrm{N}$ is a choice of even lattice of signature $(1, r)$ with $0 \le r \le 19$.
The family $\mathbf{X}_{a,b,c}$ in Theorem~\ref{maintheorem} is an example of a family of K3 surfaces
that are polarized by the even lattice $\mathrm{N} = \mathrm{H} \oplus \mathrm{E}_8 \oplus \mathrm{A}_7$ of rank seventeen.
Surfaces in this class have Picard ranks taking the four possible values 17, 18, 19 or 20.
For example, for special values of the parameters $a, b, c$ the lattice polarization extends canonically to a polarization by the unimodular rank-eighteen lattice
$ \mathrm{H} \oplus \mathrm{E}_8 \oplus \mathrm{E}_8$ (see Section \ref{coinciding_roots}).

\subsection{Two isogenies of K3 surfaces}
For the convenience of the reader we review the definition of Nikulin involution, Shioda-Inose structure, and Van Geemen-Sarti involution.
This exposition is based is based on the introductory chapters of \cite{CD, CD3}. 

Let $\mathbf{X}$ be an algebraic K3 surface over $\mathbb{C}$. A Nikulin involution is an involution $\imath$ on $\mathbf{X}$ such that
$\imath^* \omega =\omega$ for all $\omega \in H^{2,0}(\mathbf{X})$. If a Nikulin involution $\imath$ exists on $\mathbf{X}$, then it has 
exactly eight fixed points. In such a case, the quotient space is a surface with eight rational double point singularities of type $A_1$. 
The minimal resolution of this singular space is a new K3 surface, which we denote by $\mathbf{Y}$. The two K3 surfaces $\mathbf{X}$ and 
$\mathbf{Y}$ are related by a (generically) two-to-one rational map $\mathrm{pr}: \mathbf{X} \dasharrow \mathbf{Y}$
with a branch locus given by eight disjoint rational curves (the even eight configuration in the sense of Mehran \cite{Mehran}).
A special kind of Nikulin involution is a Shioda-Inose structure:
\begin{definition}
A K3 surface $\mathbf{X}$ admits a Shioda-Inose structure if there is a Nikulin involution with a rational quotient map $\mathrm{pr}:
\mathbf{X} \dashrightarrow \mathbf{Y}$ where $\mathbf{Y}$ is a Kummer surface, and the map $\mathrm{pr}_*$ induces a
Hodge isometry $\mathrm{T}_{\mathbf{X}}(2) \cong \mathrm{T}_{\mathbf{Y}}$.
\end{definition}
\noindent
The following theorem is a corollary of a theorem proved in \cite{Shioda} by Shioda in the case of an Abelian surface 
and in \cite{Kulikov,Todorov,Looijenga,Siu,Namikawa} in the case of an K3 surface:
\begin{theorem}[Morrison \cite{Morrison84}]
\label{Morrison}
\begin{enumerate}
\item[]
\item If $\mathrm{T} \hookrightarrow H^3$ (resp. $\mathrm{T} \hookrightarrow \mathrm{H}^3 \oplus \mathrm{E}_8(-1)^2$) is a primitive sub-lattice
of signature $(2,4-\rho)$ (resp. $(2,20-\rho)$), then there exists an Abelian surface (resp. algebraic K3 surface) $\mathbf{Y}$ and
an isometry  $\mathrm{T}_{\mathbf{Y}} \xrightarrow{\sim} \mathrm{T}$. 
\item An algebraic K3  surface $\mathbf{Y}$ admits a Shioda-Inose structure if and only if there is a primitive embedding $\mathrm{T}_{\mathbf{Y}} \hookrightarrow \mathrm{H}^3$.
\end{enumerate}
\end{theorem}
More generally, K3 surfaces that admit rational double covers in the Picard rank 17
onto Kummer surfaces were studied by Mehran in her Michigan PhD thesis \cite{Mehran}. 
The classes of all K3 surfaces with such a degree-two rational map contains the class of all K3
surfaces with a Shioda-Inose structure.  The remaining 372 classes of K3 surfaces in
Mehran's lattice-theoretic classification \cite{Mehran2} contain ten possibilities associated with the generic 
transcendental lattice being $\mathrm{H} \oplus \mathrm{H} \oplus \langle -8 \rangle$.
The family $\mathbf{X}_{a,b,c}$ for a generic choice choice of the parameters $a, b, c$ in Theorem~\ref{maintheorem} 
will serve as an explicit construction of one of these ten possibilities.

The Nikulin involution on $\mathbf{X}_{a,b,c}$ will be constructed by fiber-wise translations by a section of
order two in the Jacobian elliptic fibration. This class of involutions was discussed by Van Geemen and Sarti \cite{VanGeemenSarti}:
\begin{definition}
 A Van Geemen-Sarti involution is a Nikulin involution $\imath$ on $\mathbf{X}$ for which there exists a triple
 $(\varphi_{\mathbf{X}}, \hat{O}, \hat{\sigma})$ such that:
 \begin{enumerate}
 \item $\varphi_{\mathbf{X}}: \mathbf{X} \to \mathbb{CP}^1$ is an elliptic fibration on $\mathbf{X}$,
 \item $\hat{O}$ and $\hat{\sigma}$ are disjoint sections of $\varphi_{\mathbf{X}}$,
 \item $\hat{\sigma}$ is an element of order two in the Mordell-Weil group $\mathrm{MW}(\mathbf{X})$,
 \item $\imath$ is the involution obtained by extending the fiber-wise translations by $\hat{\sigma}$ in the smooth fibers of $\varphi$ using the group structure with neutral element $\hat{O}$.
 \end{enumerate}
 \end{definition}
 
\subsection{Kummer surfaces associated with curves of genus two}
\label{Kummer_surfaces}
Let us describe the Kummer surfaces of Jacobian varieties of genus-two curves in more detail.
The following well-known theorem \cite{Nikulin} explains the properties of the Jacobian variety:
\begin{theorem}
\label{Nikulin}
\begin{enumerate}
\item[]
\item Suppose that an irreducible non-singular curve $\mathbf{C}$ is embedded in an Abelian variety of dimension two. If $\mathbf{C}\cdot\mathbf{C}=2$
then $\mathbf{C}$ is a curve of genus two.
\item Let $\mathbf{C}$ be a curve of genus two. Then there exists a unique pair $(J(\mathbf{C}),j_{\mathbf{C}})$ where $J(\mathbf{C})$ is an Abelian variety
of dimension two and $j_{\mathbf{C}}:\mathbf{C} \to J(\mathbf{C})$ is an embedding. The uniqueness is understood in the following way: if there exists another
such pair $(J(\mathbf{C})',j'_{\mathbf{C}})$, then there exists a unique isomorphism $\mathbf{b}: J(\mathbf{C}) \to J(\mathbf{C})'$ such $j'_{\mathbf{C}} = 
\mathbf{b} \circ j_{\mathbf{C}}$. $J(\mathbf{C})$ is the Jacobian of the curve $\mathbf{C}$.
\item We can regain $\mathbf{C}$ from the pair $(J(\mathbf{C}),\mathcal{O}_{\mathbf{C}})$.
\item We can regain $\mathbf{C}$ from the pair $(J(\mathbf{C}),E)$, where $E=[\mathbf{C}]$ is the class of $\mathbf{C}$ in the N\'eron-Severi group
$\mathrm{NS}(J(\mathbf{C}))$.
\end{enumerate}
\end{theorem}

\begin{definition}
A curve $\mathbf{C}$ of genus-two is called generic if $\mathrm{NS}(J(\mathbf{C}))=\mathbb{Z}[\mathbf{C}]$.
\end{definition}

Let us briefly review a well-known geometric construction of the Kummer surface of the Jacobian variety of a curve of genus two.
Let $\mathrm{K}$ be a quartic surface in $\mathbb{CP}^3$ with $16$ ordinary double points on it and $\widehat{\mathrm{K}}$
its minimal resolution. These singular points are called nodes. There are $16$ hyperplanes in $\mathbb{CP}^3$ which 
are tangent to $\mathrm{K}$ along a smooth conic. Such a conic is called a trope. Each trope passes through
exactly six nodes and each node lies on exactly six tropes. This configuration of nodes and tropes is called the
$(16,6)$-configuration. We will use the notation of \cite[Sec. 4]{Kumar}.

We refer to the $16$ non-singular rational curves on $\widehat{\mathrm{K}}$ lying over the $16$ nodes
and the transforms of the $16$ tropes of $\mathrm{K}$ as nodes and tropes of $\widehat{\mathrm{K}}$ respectively.
A Kummer quartic $\mathrm{K}$ in $\mathbb{CP}^3$ can be realized as
\begin{equation}
\label{kummer}
 K(z_1,z_2,z_3,z_4)= K_2 z_4^2 + K_1 z_4 + K_0 = 0
\end{equation}
with the coefficients $K_0, K_1, K_2$ given in terms of six parameters $\theta_i$ and $z_1, z_2, z_3$.
We will use the standard notation for the labels of the $16$ nodes $p_0, \, p_{ij}$ with $1\le i < j \le 6$
and tropes $T_i, T_{ijk}$ with $1\le i < j < k \le 6$
(as described in \cite[Sec. 4]{Kumar}). The node $p_0$ is then located at $[0:0:0:1]$. 
Each of the tropes $T_i$ contains the $6$ nodes $p_0$ and $p_{ij}$ and is given by
\begin{equation}
\label{tropes}
 T_i: \theta_i^2 z_1 - \theta_i  z_2 + z_3 =0\;.
\end{equation}
The remaining $10$ tropes $T_{ijk}$ correspond to partitions of $\{1,\dots,6\}$ into two sets of three.
They contain the nodes with two-figure symbols whose digits are out either of the two sets of three in all possible ways.

A double cover of a trope, branched along the six nodes it passes through, is a curve of genus two.
Throughout this article, we will take the genus-two curve $\mathbf{C}$ to be
\begin{equation}
\label{curve_2}
 \mathbf{C}: \; y^2 = f(x) = a_0 \prod_{i=1}^6 (x-\theta_i) \;.
\end{equation}
The Jacobian variety $J(\mathbf{C})$ of the curve $\mathbf{C}$ is birational to the symmetric product of two copies of $\mathbf{C}$, i.e., 
$$(\mathbf{C}\times\mathbf{C})/\{ \mathbb{I}, \pi \}$$ with
$\pi(x_1)=x_2$ and $\pi(y_1)=y_2$. Its function field is the sub-field of $\mathbb{C}[x_1,x_2,y_1,y_2]$ with $y_i^2 =f(x_i)$ for $i=1,2$ which is fixed under 
$\pi$.
The Kummer surface $\mathrm{Kum}(J(\mathbf{C}))$ is birational to the quotient $J(\mathbf{C})/\{ \mathbb{I}, -\mathbb{I} \}$ 
with $-\mathbb{I}(x_i)=x_i$ and $-\mathbb{I}(y_i)=-y_i$ for $i=1,2$. Its function
field is the sub-field of $\mathbb{C}[x_1,x_2,y_1,y_2]$ with $y_i^2 =f(x_i)$ for $i=1,2$ which is fixed under $\pi$ and $-\mathbb{I}$. 
Thus, the function field is generated by $\eta=y_1y_2/a_0$, $\xi=x_1x_2$, and $\zeta=x_1+x_2$ with the relation
\begin{equation}
\label{kummer2}
 \eta^2 = \prod_{i=1}^6 \left( \xi - \theta_i \, \zeta + \theta_i^2 \right) \;.
\end{equation}
By setting $\xi=\frac{z_3}{z_1}$, $\zeta=\frac{z_2}{z_1}$, $\eta=\frac{y}{2z_1^3}$, Equation (\ref{kummer2}) becomes
\begin{equation}
\label{quartic}
 y^2 = 4 \, T_1 T_2 T_3 T_4 T_5 T_6 \;,
\end{equation}
which is the same as Equation~(\ref{kummer}) after completing the square
\begin{equation}
\label{quartic2}
 (K_2 z_4 + \frac{K_1}{2})^2 = \frac{K_1^4}{4} - K_0 K_2  = 4 \, T_1 T_2 T_3 T_4 T_5 T_6
\end{equation}
and setting $y=K_2 z_4 + \frac{K_1}{2}$. Thus, the surface $\widehat{\mathrm{K}}$ is in fact isomorphic to the Kummer 
surface for the Jacobian variety of the curve $\mathbf{C}$ of genus $2$ defined by Equation~(\ref{curve_2}).

The N\'eron-Severi lattice of the non-singular Kummer surface $\widehat{\mathrm{K}}$ contains rational curves
coming from the nodes $p_0, p_{ij}$ and the tropes $T_i, T_{ijk}$. We will denote the lattice generated by these curves as $\mathrm{S}$.
$\mathrm{S}$ is the even lattice of signature $(1,16)$ and $\det(\mathrm{S})=2^6$ described in \cite[Def. 6.2]{Nikulin}.
We have the following theorem:

\begin{theorem}[Nikulin \cite{Nikulin}]
A K3 surface $\mathbf{Y}$ is the Kummer surface of the Jacobian variety of a generic curve $\mathbf{C}$ of genus two 
if and only if $\mathrm{NS}(\mathbf{Y}) \cong \mathrm{S}$.
\end{theorem}

\subsection{Jacobian surfaces}
We want to restrict ourselves further to Jacobian elliptic K3 surfaces, i.e., elliptic surfaces with a section. 
For elliptically fibered surfaces with a section, the two classes in 
$\mathrm{NS}(\mathbf{Y})$ associated with the elliptic fiber and section span a sub-lattice $\mathcal{H}$ isometric to the standard hyperbolic lattice $\mathrm{H}$
of rank two. The sub-lattice $\mathcal{H} \subset \mathrm{NS}(\mathbf{Y})$ completely determines
the elliptic fibration with section on $\mathbf{Y}$. In fact, on a given K3 surface $\mathbf{Y}$ there is a one-to-one correspondence between sub-lattices $\mathcal{H} \subset \mathrm{NS}(\mathbf{Y})$ isometric to the standard hyperbolic lattice $\mathrm{H}$ that contain a pseudo-ample class and elliptic structures with section on $\mathbf{Y}$ which realize $\mathcal{H}$ \cite[Thm. 2.3]{CD2}. 

The distinct ways up to isometries to embed the standard rank-$2$ hyperbolic lattice $\mathrm{H}$ isometrically $\mathrm{NS}(\mathbf{Y})$ are distinguished by the isomorphism type of the orthogonal complement $\mathcal{W}$ of $\mathcal{H}$,
such that the N\'eron-Severi lattice decomposes as a direct orthogonal sum
\begin{equation*}
 \mathrm{NS}(\mathbf{Y}) = \mathcal{H} \oplus \mathcal{W} \;,
\end{equation*}
%The orthogonal complement $\mathcal{W}$ has rank $15$ and is an even, negative definite lattice with discriminant $2^6$.
The sub-lattice $\mathcal{W}^{\mathrm{root}} \subset \mathcal{W}$ is spanned by the roots, i.e., the algebraic classes of
self-intersection $-2$ inside $\mathcal{W}$. The singular fibers of the elliptic fibration determine $\mathcal{W}^{\mathrm{root}}$
uniquely up to permutation. Moreover, there exists a canonical group isomorphism \cite{Miranda}
\begin{equation}
 \mathcal{W}/\mathcal{W}^{\mathrm{root}} \stackrel{\sim} \longrightarrow \mathrm{MW}(\mathbf{Y}) \;,
\end{equation}
where $\mathrm{MW}(\mathbf{Y})$ is group of sections on $\mathbf{Y}$ compatible with its elliptic structure.
 
\section{Two elliptic fibrations of Picard rank $17$}
\label{K3families}
In this section, we will define two 3-parameter families $\mathbf{X}_{a,b,c}$ and $\mathbf{Y}_{a,b,c}$
of Jacobian elliptic K3 surfaces. We show that each surface $\mathbf{X}_{a,b,c}$ has a Nikulin involution 
whose quotient is $\mathbf{Y}_{a,b,c}$.
 
\subsection{The family $\mathbf{Y}_{a,b,c}$}
\noindent
A $3$-parameter family of Jacobian elliptic surfaces $\mathbf{Y}_{a,b,c} \;$ over $\mathbb{CP}^1$ is defined by
\begin{equation}
\label{WE}
\mathbf{Y}_{a,b,c} = \left\lbrace (Y,X,t) \in \mathbb{C}^3 \mid Y^2 = 4 \, X^3 - g_2 \, X - g_3 \right\rbrace \;.
\end{equation}
Here, $t$ is the affine coordinate on the base curve $\mathbb{CP}^1$ and
\begin{eqnarray*}
g_2  & = & \; \frac{1}{3} \; p^2(t) + r^4(t) \;,\\
g_3  & = & \frac{1}{27} \; p(t) \, \Big(p(t) - 3  \, r^2(t) \Big) \, \Big(p(t) + 3 \, r^2(t) \Big) \;,
\end{eqnarray*}
and the discriminant $\Delta = g_2^3 -27 \, g_3^2$ equals
\begin{eqnarray}
\label{Ydiscriminant}
\Delta  & = & r^4(t)  \, \Big(p(t) - r^2(t)\Big)^2 \, \Big(p(t) + r^2(t)\Big)^2 \;.
\end{eqnarray}
The degrees of the polynomials $g_2$ and $g_3$ make $\mathbf{Y}_{a,b,c}$ a family of K3 surfaces.
A holomorphic symplectic two-form is $\omega = dt\wedge dX/Y$.
By setting $X = \frac{1}{4} x + \frac{1}{3} P$ and $Y=\frac{1}{4}y$, the Weierstrass equation in (\ref{WE}) becomes
\begin{equation*}
\label{we}
y^2 =  x \;  \Big(x + 2 \, p_+(t) \Big) \Big(x + 2\, p_-(t)\Big) \;,
\end{equation*}
whence the Mordell-Weil group is $\mathrm{MW}(\mathbf{Y})=(\mathbb{Z}_2)^2$.
Under $t \mapsto -t$ and $b \mapsto -b$, $c \mapsto -c$, we have that 
$g_2 \mapsto  g_2$ and $g_3 \mapsto  - g_3$. Hence, we have
\begin{equation}
 \mathbf{Y}_{a,b,c} \cong  \mathbf{Y}_{a,-b,-c} \;.
\end{equation}
Notice that the positions of the roots in Equations (\ref{Xdiscriminant}) and (\ref{Ydiscriminant}) 
coincide. Thus, the notion of genericity of $\mathbf{Y}_{a,b,c}$ and $\mathbf{X}_{a,b,c}$ agrees.
For generic values of $a, b, c$, the surfaces $\mathbf{Y}_{a,b,c}$ have the following 
configuration of singular fibers:

\bigskip

\begin{tabular}{l|c|c|c|c|c|l}
locus   & number of points & $\nu(g_2)$ & $\nu(g_3)$ & $\nu(\Delta)$ & Kodaira type of fiber & $\mathcal{W}^{\mathrm{root}}$ \\
\hline
$r = 0$      & $1$ & $0$ & $0$ & $4$ & $I_4$  & $\mathrm{A}_3$ \\
$p_+ =0$ & $3$ & $0$ & $0$ & $2$ & $I_2$      & $\mathrm{A}_1$ \\
$p_- =0$ & $3$ & $0$ & $0$ & $2$ & $I_2$      & $\mathrm{A}_1$ \\
$t = \infty$ & $1$ & $2$ & $3$ & $8$ & $I_2^*$& $\mathrm{D}_6$ 
\end{tabular}

\bigskip

\noindent
In this case the N\'eron-Severi lattice $\mathrm{NS}(\mathbf{Y})$ has signature $(1,16)$ and the discriminant $(4)(2^6)(2\cdot 2)/4^2=2^6$.

\subsubsection{Coinciding roots}
By setting $t=\mu \,\tilde{t}$, $g_2 = \lambda^4 \, \tilde{g}_2$, $g_3 = \lambda^6\, \tilde{g}_3$,
$p = \lambda^2 \, \tilde{p}$, $r = \lambda \, \tilde{r}$ with $\mu^3=c^2$, $\lambda=-c$, we obtain
\begin{equation*}
\tilde{p} =  4 \, \tilde{t}^{\, 3} - 3  \,  \tilde{a}  \, \tilde{t} - \tilde{b} \,, \quad
\tilde{r} =  1 - \tilde{c} \, \tilde{t}\;, \quad \tilde{\Delta}  =  \tilde{r}^4\left(\tilde{t}\,\right)  \, \Big(\tilde{p}^2\left(\tilde{t}\,\right) 
- \tilde{r}^4\left(\tilde{t}\,\right)\Big)^2 \;,
\end{equation*}
with
\begin{equation*}
\tilde{a} =  \frac{a}{\mu^2} \;, \quad
\tilde{b} =  \frac{b}{c^2} \;, \quad
\tilde{c} =  \frac{\mu}{c} \;.
\end{equation*}
The case $\tilde{c}=0$ was considered by Clingher and Doran \cite{CD}: for generic values $\tilde{a}, \tilde{b}$ the surface $\mathbf{\tilde{Y}}_{\tilde{a},\tilde{b},0}$ has the singular fibers
$I_6^* \oplus 6 \, I_2$ and $\mathrm{MW}(\mathbf{\tilde{Y}})=(\mathbb{Z}_2)^2$. This configuration was denoted by $\mathfrak{J}_5$ 
in the Oguiso classification of Jacobian fibrations on Kummer surfaces of the product of nonisogenous elliptic curves \cite{Oguiso}. 
The N\'eron-Severi lattice $\mathrm{NS}(\mathbf{\tilde{Y}})$ has signature $(1,17)$ and discriminant $(2\cdot 2)(2^6)/4^2=2^4$.

When $c$ coincides with one of the roots $t_i$ of $p_{\pm}$, the surface $\mathbf{Y}_{a,b,c}$ has the singular fibers
$2 \, I_2^* \oplus 4 \, I_2$ and $\mathrm{MW}(\mathbf{Y})=(\mathbb{Z}_2)^2$. This configuration was denoted by $\mathfrak{J}_6$
in the Oguiso classification of Jacobian fibrations on Kummer surfaces of the product of nonisogenous elliptic curves. 
The N\'eron-Severi lattice $\mathrm{NS}(\mathbf{Y})$ has signature $(1,17)$ and discriminant $(2\cdot 2)^2(2^4)/4^2=2^4$. 

\subsection{The family $\mathbf{X}_{a,b,c}$}
\noindent
A second $3$-parameter family of Jacobian elliptic surfaces $\mathbf{X}_{a,b,c} \;$ over $\mathbb{CP}^1$ is defined by
\begin{equation}
\label{WE2}
\mathbf{X}_{a,b,c} = \left\lbrace (\hat{Y},\hat{X},t) \in \mathbb{C}^3 \mid \hat{Y}^2 = 4 \, \hat{X}^3 - \hat{g}_2 \, \hat{X} - \hat{g}_3 \right\rbrace \;,
\end{equation}
Here, $t$ is the affine coordinate on the base curve $\mathbb{CP}^1$ and
\begin{eqnarray*}
\hat{g}_2     & = & \; \frac{1}{3} \, p^2(t) - \frac{1}{4}\, r^4(t) \;,\\
\hat{g}_3     & = & \frac{1}{216} \, p(t) \, \Big( 8 \, p^2(t) - 9 \, r^4(t)  \Big) \;,
\end{eqnarray*}
and discriminant $\hat{\Delta} = \hat{g}_2^3 -27 \, \hat{g}_3^2$ with
\begin{eqnarray}
\label{Xdiscriminant}
\hat{\Delta}  & = & \frac{1}{64} \, r^8(t)  \, \Big(p(t) - r^2(t)\Big) \, \Big(p(t) + r(t)^2\Big) \;.
\end{eqnarray}
The degree of the polynomials $\hat{g}_2$ and $\hat{g}_3$ make $\mathbf{X}_{a,b,c}$ a family of K3 surfaces.
A holomorphic symplectic two-form is $\hat{\omega} =  dt\wedge d\hat{X}/\hat{Y}$.
By choosing $\hat{X}= \frac{1}{4} \hat{x} - \frac{1}{6} \, P$ and $\hat{Y} = \frac{1}{4} \, \hat{y}$, 
the Weierstrass equation in (\ref{WE2}) becomes
\begin{equation*}
  \hat{y}^2 = \hat{x} \, \Big(\hat{x}^2 - 2 \, p(t) \, \hat{x} + r^4(t) \Big) \;,
\end{equation*} 
whence the Mordell-Weil group is $\mathrm{MW}(\mathbf{X})=\mathbb{Z}_2$.
The neutral element  and a two-torsion section are given by  $\hat{O}: (\hat{Y},\hat{X})=(0,0)$ and $\hat{\sigma}: (\hat{Y},\hat{X})=(0,-p(t)/6)$, respectively.
Under $t \mapsto -t$ and $b \mapsto -b$, $c \mapsto -c$, we have that 
$\hat{g}_2 \mapsto  \hat{g}_2$ and $\hat{g}_3 \mapsto  - \hat{g}_3$. Hence we have
\begin{equation}
 \mathbf{X}_{a,b,c} \cong  \mathbf{X}_{a,-b,-c} \;.
\end{equation}
For generic values of $a, b, c$, the surface $\mathbf{X}_{a,b,c}$ has the following 
configuration of singular fibers:

\bigskip

\begin{tabular}{l|c|c|c|c|c|l}
locus   & number of points & $\nu(g_2)$ & $\nu(g_3)$ & $\nu(\Delta)$ & Kodaira type of fiber & $\mathcal{W}^{\mathrm{root}}$ \\
\hline
$r = 0$      & $1$ & $0$ & $0$ & $8$  & $I_8$  & $\mathrm{A}_7$ \\
$p_+ =0$ & $3$ & $0$ & $0$ & $1$  & $I_1$  & $-$ \\
$p_- =0$ & $3$ & $0$ & $0$ & $1$  & $I_1$  & $-$ \\
$t = \infty$ & $1$ & $2$ & $3$ & $10$ & $I_4^*$& $\mathrm{D}_8$ 
\end{tabular}

\bigskip

\noindent
We have the following lemma:
\begin{lemma}
\label{parts1}
The Jacobian elliptic surface $\mathbf{X}_{a,b,c} \to \mathbb{CP}^1$ 
is a  K3 surface with a $ \langle 8 \rangle  \oplus \mathrm{E}_8  \oplus \mathrm{E}_8$-lattice polarization.
\end{lemma}

\begin{proof}
The degree of the polynomials $\hat{g}_2$ and $\hat{g}_3$ make $\mathbf{X}_{a,b,c}$ a family of K3 surfaces.
For generic values of $a, b, c$,
the N\'eron-Severi lattice $\mathrm{NS}(\mathbf{X})$ has signature $(1,16)$ and discriminant 
$(2\cdot 2) \, (8)/4=2^3$. In fact, we show in Appendix \ref{discriminant2} that in this case $\mathrm{NS}(\mathbf{X}) \cong \langle 8 \rangle  \oplus \mathrm{E}_8  \oplus \mathrm{E}_8$ and
$\mathrm{T}_{\mathbf{X}} =  \mathrm{H}^2 \oplus \langle -8\rangle$. 
\end{proof}

\subsubsection{Coinciding roots}
\label{coinciding_roots}
In the case where $\tilde{c}=0$ but $\tilde{a}, \tilde{b}$ are generic, the configuration of singular fibers becomes $I_{12}^* \oplus 6 \, I_1$.
The N\'eron-Severi lattice $\mathrm{NS}(\mathbf{\tilde{X}})$ has signature $(1,16)$ and discriminant 
$(2\cdot 2)/4=1$. In fact, we show $\mathrm{NS}(\mathbf{\tilde{X}}) \cong \mathrm{H} \oplus \mathrm{D}_{16}^+$
and $\mathrm{T}_{\mathbf{\tilde{X}}} =  \mathrm{H}^2 $ in Appendix \ref{discriminant1}. In fact, it was shown in \cite{CD} that
$\tilde{\mathbf{X}}$ is polarized by the unimodular rank-eighteen lattice $\mathrm{H} \oplus \mathrm{E}_8 \oplus \mathrm{E}_8$.

\subsection{The Nikulin involution}
\label{VGSinvolution}
The relation between the K3 surfaces $\mathbf{X}_{a,b,c}$ and $\mathbf{Y}_{a,b,c}$ is given by
the following lemma:
\begin{lemma}
\label{parts2}
The Jacobian elliptic K3 surface $\mathbf{X}_{a,b,c}$ admits a Van Geemen-Sarti involution such that
the induced quotient map is a rational double cover of $\mathbf{Y}_{a,b,c}$.
\end{lemma}
\begin{proof}
The Jacobian elliptic surfaces $\mathbf{X}_{a,b,c} \to \mathbb{CP}^1$ and $\mathbf{Y}_{a,b,c}\to\mathbb{CP}^1$ are related by fiberwise two-isogeny.
To see this let $\hat{E}_t$ be the elliptic fiber of $\mathbf{X}_{a,b,c}$ over $t$, and  $E_t$ the fiber of $\mathbf{Y}_{a,b,c}$.
We define the maps $\jmath: \hat{E}_t \to E_t$ and $\jmath^{\; \prime}: E_t \to \hat{E}_t$ by setting
\begin{equation}
\label{transfoK3}
 \begin{array}{rrclcrcl}
 \jmath:      &  x       & = & \dfrac{\hat{y}^2}{\hat{x}^2}\;, & \; & y       &=& 
 \dfrac{\hat{y} \, (\hat{x}^2 - M^4)}{\hat{x}^2}\;, \\
 \jmath^{\; \prime}: &  \hat{x} & = & \dfrac{y^2}{4 \, x^2} \;,       & \; & \hat{y} &=& 
 \dfrac{y \, (x^2 - P^2 + 4 \, M)}{8 \, x^2} \;.
 \end{array} 
\end{equation}
If $\hat{E}_t$ is the elliptic fiber of $\mathbf{X}_{a,b,c}$ over $t$, and  $E_t$ the fiber of $\mathbf{Y}_{a,b,c}$, then
the map $\jmath$ corresponds to the two-isogeny $E_t=\hat{E}_t/\{\hat{O},\hat{\sigma}\}$ given by the sections $\hat{O}$ and $\hat{\sigma}$ that
defines the neutral element and a two-torsion point in the fiber $\hat{E}_t$, respectively. 
The computation is carried out in detail in \cite[Sec. 4.5]{Husemoeller}.
Hence, we obtain a rational quotient map of degree two
\begin{equation}
\label{K3_quotient_map}
 \mathrm{pr}: \mathbf{X}_{a,b,c} \dashrightarrow  \mathbf{Y}_{a,b,c} \;.
\end{equation}
One checks that $\jmath^*$ maps the holomorphic two-form $\hat{\omega}$ on $\mathbf{X}_{a,b,c}$
to $\omega$ on $\mathbf{Y}_{a,b,c}$ since 
\begin{equation*}
 \omega=dt \wedge \frac{dX}{Y} = dt \wedge \frac{dx}{y} = \jmath^*\Big( dt \wedge \frac{d\hat{x}}{\hat{y}} \Big) =  
 \jmath^* \Big( dt \wedge \frac{d\hat{X}}{\hat{Y}}\Big)=\jmath^*\hat{\omega}\;,
\end{equation*}
and similarly for $(\jmath^{\; \prime})^*$. Thus, $\imath=\jmath^{\; \prime}\circ\jmath$ is a Nikulin involution on $\mathbf{X}_{a,b,c}$. 
It follows that $(\mathbf{X}_{a,b,c} \to \mathbb{CP}^1, \hat{O}, \hat{\sigma})$ is a Van Geemen-Sarti involution compatible with the
Nikulin involution $\imath$.
\end{proof}

\section{The relation to the $\mathrm{SU}(2)$ Seiberg-Witten curves}
\label{SeibergWitten_curves}
In this section, we will define the Seiberg-Witten curves $\mathbf{X}_{\scriptscriptstyle \mathrm{SW}}$ and $\mathbf{Y}_{\scriptscriptstyle \mathrm{SW}}$ 
of $\mathcal{N}=2$ supersymmetric $\mathrm{SU}(2)$-Yang-Mills theory with $N_f=2$ hypermultiplets and $N_f=0$ respectively. We also show that the families of 
elliptic K3 surfaces $\mathbf{X}_{a,b,c}$, $\mathbf{Y}_{a,b,c}$, and the Nikulin involution $\imath$ are obtained from $\mathbf{X}_{\scriptscriptstyle \mathrm{SW}}$, $\mathbf{Y}_{\scriptscriptstyle \mathrm{SW}}$, and their two-isogeny respectively by a base change.

\subsection{The Seiberg-Witten curve for $N_f=2$.}
The Seiberg-Witten curve of $\mathcal{N}=2$ supersymmetric $\mathrm{SU}(2)$-Yang-Mills theory with $N_f=2$ hypermultiplets (see \cite{SW}) is the 
Jacobian elliptic surfaces $\mathbf{Y}_{\scriptscriptstyle \mathrm{SW}}$ over $\mathbb{CP}^1$ defined by
\begin{equation}
\label{WE3}
\mathbf{Y}_{\scriptscriptstyle \mathrm{SW}} = \left\lbrace \left(Y_{\scriptscriptstyle \mathrm{SW}},X_{\scriptscriptstyle \mathrm{SW}},u\right) \in \mathbb{C}^3 
\mid Y_{\scriptscriptstyle \mathrm{SW}}^2 
= 4 \, X_{\scriptscriptstyle \mathrm{SW}}^3 - g^{\scriptscriptstyle \mathrm{SW}}_2 \, X_{\scriptscriptstyle \mathrm{SW}} - g^{\scriptscriptstyle \mathrm{SW}}_3 \right\rbrace \;.
\end{equation}
Here, $u$ is the affine coordinate on the base curve $\mathbb{CP}^1$ and
\begin{eqnarray*}
g^{\scriptscriptstyle \mathrm{SW}}_2  & = & \frac{1}{3} \, u^2 + 1 \;,\\
g^{\scriptscriptstyle \mathrm{SW}}_3  & = & \frac{1}{27} \, u \, \Big(u^2 - 9 \Big) \;,\\
\Delta^{\scriptscriptstyle \mathrm{SW}}  & = & \, \Big(u^2 - 1 \Big)^2 \;.
\end{eqnarray*}
The degrees of the polynomials $g_2^{\scriptscriptstyle \mathrm{SW}}$ and $g_3^{\scriptscriptstyle \mathrm{SW}}$ make $\mathbf{Y}_{\scriptscriptstyle \mathrm{SW}}$ a rational elliptic surface.
A holomorphic symplectic two-form is given by $\omega_{\scriptscriptstyle \mathrm{SW}} = du\wedge dX_{\scriptscriptstyle \mathrm{SW}}/Y_{\scriptscriptstyle \mathrm{SW}}$.
By setting $X_{\scriptscriptstyle \mathrm{SW}} = \frac{1}{4} x_{\scriptscriptstyle \mathrm{SW}} + \frac{1}{3} u$ and $Y_{\scriptscriptstyle \mathrm{SW}}=\frac{1}{4}y_{\scriptscriptstyle \mathrm{SW}}$, the Weierstrass equation in (\ref{WE3}) becomes
\begin{equation*}
y_{\scriptscriptstyle \mathrm{SW}}^2 =  x_{\scriptscriptstyle \mathrm{SW}} \;  \left(x_{\scriptscriptstyle \mathrm{SW}} + 2 \, u - 2 \right) \left(x_{\scriptscriptstyle \mathrm{SW}} + 2\, u + 2\right) \;,
\end{equation*}
whence the Mordell-Weil group is $\mathrm{MW}(\mathbf{Y}_{\scriptscriptstyle \mathrm{SW}})=(\mathbb{Z}_2)^2$.
The surface $\mathbf{Y}_{\scriptscriptstyle \mathrm{SW}}$ has the following 
configuration of singular fibers:

\bigskip
\begin{center}
\begin{tabular}{l|c|c|c|c|l}
locus   & $\nu(g^{\scriptscriptstyle \mathrm{SW}}_2)$ & $\nu(g^{\scriptscriptstyle \mathrm{SW}}_3)$ & $\nu(\Delta^{\scriptscriptstyle \mathrm{SW}})$ & Kodaira type of fiber & $\mathcal{W}^{\mathrm{root}}$ \\
\hline
$u = 1$      & $0$ & $0$ & $2$  & $I_2$  & $\mathrm{A}_1$ \\
$u = -1$     & $0$ & $0$ & $2$  & $I_2$  & $\mathrm{A}_1$ \\
$u = \infty$ & $2$ & $3$ & $8$  & $I_2^*$& $\mathrm{D}_6$ 
\end{tabular}
\end{center}

\bigskip
The rational elliptic surface $\mathbf{Y}_{\scriptscriptstyle \mathrm{SW}}$ is in fact the modular elliptic surface over the base curve $\mathbb{H}/\Gamma(2)$ where $\mathbb{H}$ is the upper half-plane and $\Gamma(2)$ is the principal congruence subgroup of level $2$ in $\mathrm{SL}(2,\mathbb{Z})$.  The relation between the affine coordinate $u \in \mathbb{CP}^1$ and $\tau \in \mathbb{H}/\Gamma(2)$ is given by
\begin{equation}
\label{u_coordinate2}
 u = \dfrac{\vartheta_2^4\left(2\,\tau\right) + \vartheta_3^4\left(2 \, \tau \right)}{2 \; \vartheta_2^2\left(2 \, \tau \right)\, \vartheta_3^2\left(2 \, \tau\right)} \;,
\end{equation}
where $\vartheta_2, \vartheta_3$ are the usual Jacobi $\vartheta$-functions. 
Hence, the elliptic fiber $E_u$ of $\mathbf{Y}_{\scriptscriptstyle \mathrm{SW}}$ has an elliptic $\tau$-parameter such that Equation~(\ref{u_coordinate2}) holds.

\subsection{The Seiberg-Witten curve for $N_f=0$.}
The Seiberg-Witten curve of $\mathcal{N}=2$ supersymmetric $\mathrm{SU}(2)$-Yang-Mills theory with $N_f=0$ (see \cite{SW}) is the 
Jacobian elliptic surface $\mathbf{X}_{\scriptscriptstyle \mathrm{SW}}$ over $\mathbb{CP}^1$ defined by
\begin{equation}
\label{WE4}
\mathbf{X}_{\scriptscriptstyle \mathrm{SW}} = \left\lbrace \left(\hat{Y}_{\scriptscriptstyle \mathrm{SW}},\hat{X}_{\scriptscriptstyle \mathrm{SW}},u\right) \in \mathbb{C}^3 \mid \hat{Y}_{\scriptscriptstyle \mathrm{SW}}^2 
= 4 \, \hat{X}_{\scriptscriptstyle \mathrm{SW}}^3 - \hat{g}^{\scriptscriptstyle \mathrm{SW}}_2 \, \hat{X}_{\scriptscriptstyle \mathrm{SW}} - \hat{g}^{\scriptscriptstyle \mathrm{SW}}_3 \right\rbrace \;.
\end{equation}
Here, $u$ is the affine coordinate on the base curve $\mathbb{CP}^1$ and
\begin{eqnarray*}
\hat{g}^{\scriptscriptstyle \mathrm{SW}}_2  & = & \frac{1}{3} \, u^2 - \frac{1}{4} \;,\\
\hat{g}^{\scriptscriptstyle \mathrm{SW}}_3  & = & \frac{1}{216} \, u \, \Big(8 \, u^2 - 9 \Big) \;,\\
\Delta^{\scriptscriptstyle \mathrm{SW}}  & = & \frac{1}{64}  \, \Big(u^2 - 1 \Big) \;.
\end{eqnarray*}
The degrees of the polynomials $\hat{g}^{\scriptscriptstyle \mathrm{SW}}_2$ and $\hat{g}^{\scriptscriptstyle \mathrm{SW}}_3$ make $\mathbf{X}_{\scriptscriptstyle \mathrm{SW}}$ a rational elliptic surface.
A holomorphic symplectic two-form is $\hat{\omega}_{\scriptscriptstyle \mathrm{SW}} = du\wedge d\hat{X}_{\scriptscriptstyle \mathrm{SW}}/\hat{Y}_{\scriptscriptstyle \mathrm{SW}}$.
By choosing $\hat{X}_{\scriptscriptstyle \mathrm{SW}}= \frac{1}{4} \hat{x}_{\scriptscriptstyle \mathrm{SW}} - \frac{1}{6} \, u$ and $\hat{Y}_{\scriptscriptstyle \mathrm{SW}} 
= \frac{1}{4} \, \hat{y}_{\scriptscriptstyle \mathrm{SW}}$, 
the Weierstrass equation in (\ref{WE2}) becomes
\begin{equation*}
  \hat{y}_{\scriptscriptstyle \mathrm{SW}}^2 = \hat{x}_{\scriptscriptstyle \mathrm{SW}} \, \Big(\hat{x}_{\scriptscriptstyle \mathrm{SW}}^2 - 2 \, u \, \hat{x}_{\scriptscriptstyle \mathrm{SW}} 
+ 1 \Big) \;,
\end{equation*} 
whence the Mordell-Weil group is $\mathrm{MW}(\mathbf{X}_{\scriptscriptstyle \mathrm{SW}})=\mathbb{Z}_2$.
The neutral element and a two-torsion section are given by $\hat{O}: (\hat{Y}_{\scriptscriptstyle \mathrm{SW}},\hat{X}_{\scriptscriptstyle \mathrm{SW}})=(0,0)$ and 
$\hat{\sigma}_{\scriptscriptstyle \mathrm{SW}}: (\hat{Y}_{\scriptscriptstyle \mathrm{SW}},\hat{X}_{\scriptscriptstyle \mathrm{SW}})=(0,-u/6)$. 
The surface $\mathbf{X}_{\scriptscriptstyle \mathrm{SW}}$ has the following 
configuration of singular fibers:

\bigskip
\begin{center}
\begin{tabular}{l|c|c|c|c|l}
locus   & $\nu(g^{\scriptscriptstyle \mathrm{SW}}_2)$ & $\nu(g^{\scriptscriptstyle \mathrm{SW}}_3)$ & $\nu(\Delta^{\scriptscriptstyle \mathrm{SW}})$ & Kodaira type of fiber & $\mathcal{W}^{\mathrm{root}}$ \\
\hline
$u = 1$      & $0$ & $0$ & $1$   & $I_1$  & $-$ \\
$u = -1$     & $0$ & $0$ & $1$   & $I_1$  & $-$ \\
$u = \infty$ & $2$ & $3$ & $10$  & $I_4^*$& $\mathrm{D}_8$ 
\end{tabular}
\end{center}

\bigskip
The rational elliptic surface $\mathbf{X}_{\scriptscriptstyle \mathrm{SW}}$ is the modular elliptic surface over the base curve $\mathbb{H}/\Gamma_0(4)$ where 
$\Gamma_0(4)$ is the congruence subgroup of the modular group of upper triangular matrices in $\mathrm{SL}(2,\mathbb{Z})$ modulo $4$.
The relation between the affine coordinate $u \in \mathbb{CP}^1$ and $\hat{\tau} \in \mathbb{H}/\Gamma_0(4)$ is given by
\begin{equation}
\label{u_coordinate}
 u = \dfrac{\vartheta_2^4(\hat{\tau}) + \vartheta_3^4(\hat{\tau})}{2 \; \vartheta_2^2(\hat{\tau}) \, \vartheta_3^2(\hat{\tau})} \;,
\end{equation}
where $\vartheta_2, \vartheta_3$ are the usual Jacobi $\vartheta$-functions.
Hence, the elliptic fiber $E_u$ of $\mathbf{X}_{\scriptscriptstyle \mathrm{SW}}$ has an elliptic $\hat{\tau}$-parameter such 
that Equation~(\ref{u_coordinate}) holds.

\subsection{Two-isogeny}
\noindent
The surfaces $\mathbf{X}_{\scriptscriptstyle \mathrm{SW}}$ and $\mathbf{Y}_{\scriptscriptstyle \mathrm{SW}}$ are related by fiberwise two-isogeny.
Let $\hat{E}_u$ bet the elliptic fiber of $\mathbf{X}_{\scriptscriptstyle \mathrm{SW}}$ over $u$, and  $E_u$ the fiber of $\mathbf{Y}_{\scriptscriptstyle \mathrm{SW}}$.
We define the maps $\jmath_{\scriptscriptstyle \mathrm{SW}}: \hat{E}_u \to E_u$ and $\jmath^{\; \prime}_{\scriptscriptstyle \mathrm{SW}}: E_u \to \hat{E}_u$
given by
\begin{equation}
\label{transfoSW}
 \begin{array}{rrclcrcl}
 \jmath_{\scriptscriptstyle \mathrm{SW}}:      &  x_{\scriptscriptstyle \mathrm{SW}}       & = & \dfrac{\hat{y}_{\scriptscriptstyle \mathrm{SW}}^2}{\hat{x}_{\scriptscriptstyle \mathrm{SW}}^2} \;,
& \; & y_{\scriptscriptstyle \mathrm{SW}}       &=& 
 \dfrac{\hat{y} \, (\hat{x}^2 - 1)}{\hat{x}^2} \;, \\ \\
 \jmath^{\; \prime}_{\scriptscriptstyle \mathrm{SW}}: &  \hat{x}_{\scriptscriptstyle \mathrm{SW}} & = & \dfrac{y_{\scriptscriptstyle \mathrm{SW}}^2}{4 \, x_{\scriptscriptstyle \mathrm{SW}}^2} \;,   
    & \; & \hat{y}_{\scriptscriptstyle \mathrm{SW}} &=& 
 \dfrac{y_{\scriptscriptstyle \mathrm{SW}} \, (x_{\scriptscriptstyle \mathrm{SW}}^2 - u^2 + 4 )}{8 \, x_{\scriptscriptstyle \mathrm{SW}}^2} \;.
 \end{array} 
\end{equation}
The map $\jmath_{\scriptscriptstyle \mathrm{SW}}$ corresponds to the two-isogeny $E_u=\hat{E}_u/\{\hat{O},\hat{\sigma}_{\scriptscriptstyle \mathrm{SW}}\}$ given by the section
$\hat{\sigma}_{\scriptscriptstyle \mathrm{SW}}$ that defines a two-torsion point in the fiber $\hat{E}_u$, and we have
\begin{equation}
\label{tau_parameter}
   \tau  = \frac{1}{2} \, \hat{\tau} \;.
\end{equation}

\subsection{Base change and ramification}
\label{base change}
A base change provides a method to obtain the families of Jacobian elliptic K3 surfaces $\mathbf{Y}_{a,b,c}$ and $\mathbf{X}_{a,b,c}$
from the the $\mathrm{SU}(2)$-Seiberg-Witten curves $\mathbf{Y}_{\scriptscriptstyle \mathrm{SW}}$ and $\mathbf{X}_{\scriptscriptstyle \mathrm{SW}}$
respectively. 

\begin{lemma}
\label{parts3a}
For the rational surjective map 
$f_{a,b,c}: \mathbb{CP}^1 \twoheadrightarrow \mathbb{CP}^1$ given by $t \mapsto u = p(t)/r^2(t)$ in affine coordinates, we obtain the elliptic surface $\mathbf{Y}_{a,b,c}$
from the elliptic fibration $\mathbf{Y}_{\scriptscriptstyle \mathrm{SW}} \to \mathbb{CP}^1$  by taking the fibered product via $f_{a,b,c}$:
\begin{equation*}
\begin{array}{rclcc}
 \mathbf{Y}_{a,b,c}  =  \mathbf{Y}_{\scriptscriptstyle \mathrm{{SW}}} & \times_{\mathbb{CP}_u^1}   & \mathbb{CP}_t^1 & \longrightarrow             & \mathbb{CP}_t^1  \\
                                                   & \downarrow                 &                 &                             & \phantom{f_{a,b,c}}\downarrow f_{a,b,c}\\
                                                   & \mathbf{Y}_{\scriptscriptstyle \mathrm{{SW}}} &                 & \longrightarrow             & \mathbb{CP}_u^1  
\end{array} 
\end{equation*}
The base change replaces the smooth fiber $E_u$ by smooth fibers $E_t$, i.e.,
\begin{equation}
 \tau(E_t) = \tau(E_u) \;\; \text{for all $t$ such that $u=f_{a,b,c}(t)$}\;.
\end{equation}
The same relation holds between  $\mathbf{X}_{a,b,c}$ and $\mathbf{X}_{\scriptscriptstyle \mathrm{SW}}$.
The surface $\mathbf{X}_{a,b,c}$ admits a rational map of degree two onto the surface $\mathbf{Y}_{a,b,c}$.
\end{lemma}
\noindent
We added an index $t$ and $u$ to distinguish the base curves of $\mathbf{Y}_{a,b,c}$ and $\mathbf{Y}_{\scriptscriptstyle \mathrm{{SW}}}$, and
$\mathbf{X}_{a,b,c}$ and $\mathbf{X}_{\scriptscriptstyle \mathrm{{SW}}}$, respectively.
\begin{proof}
To compute the pull-back of the Weierstrass equation (\ref{WE3}) we replace $u$ by $p(t)/r^2(t)$.
The effect of the base change on the singular fibers depends on the local ramification of the morphism $f_{a,b,c}$. The points $u=\pm 1$ are unramified points. Hence, the singular fibers of Kodaira type $I_2$ at $u=1$ and $u=-1$ in $\mathbf{Y}_{\scriptscriptstyle \mathrm{SW}}$ are replaced by three copies at 
$t_2, t_4, t_6$ and $t_1, t_3, t_5$, respectively, where $p_\pm(t)=0$.

At $u=\infty$, there is ramification. The vanishing order of the $g^{\scriptscriptstyle \mathrm{SW}}_2, g^{\scriptscriptstyle \mathrm{SW}}_3, \Delta^{\scriptscriptstyle \mathrm{SW}}$  
is multiplied by the ramification index which is two. However, the Weierstrass form is not minimal. To minimalize the fibration, i.e.,
to clear the denominators in $g_2^{\scriptscriptstyle \mathrm{SW}}$ and $g_3^{\scriptscriptstyle \mathrm{SW}}$, we set
\begin{equation}
\label{transfo}
 X_{\scriptscriptstyle \mathrm{SW}}= \dfrac{X}{r^2(t)}\;, \qquad Y_{\scriptscriptstyle \mathrm{SW}} = \dfrac{Y}{r^3(t)} \;.
\end{equation}
We obtain the Weierstrass form (\ref{WE}) of $\mathbf{Y}_{a,b,c}$ since
\begin{equation*}
 g_2(t) = r^4(t) \;\; g_2^{\scriptscriptstyle \mathrm{SW}}\left(u=\frac{p(t)}{r^2(t)}\right) \;, \qquad g_3(t) = r^6(t) \;\; g_3^{\scriptscriptstyle \mathrm{SW}}\left(u=\frac{p(t)}{r^2(t)}\right) \;.
\end{equation*}
Thus, the singular fiber of Kodaira type $I_2^*$ at $u=\infty$ in $\mathbf{Y}_{\scriptscriptstyle \mathrm{SW}}$ is replaced by a singular fiber of Kodaira type $I_2^*$ at $t=\infty$ and a singular fiber of Kodaira type $I_4$ at $r(t)=0$ in the base change. Under the base change, the holomorphic two-form $\omega_{\scriptscriptstyle \mathrm{SW}}$ is pulled back to a holomorphic two-form on $\mathbf{Y}_{a,b,c}$ which is
\begin{equation}
\label{holomorphic_form2}
 f_{a,b,c}^* \, \omega_{\scriptscriptstyle \mathrm{SW}} =  f_{a,b,c}^* \left( du \wedge \dfrac{dX_{\scriptscriptstyle \mathrm{SW}}}{Y_{\scriptscriptstyle \mathrm{SW}}} \right) = r(t) \, \dfrac{\partial u}{\partial t} \; \omega
 \;,
\end{equation}
with $\partial u/\partial t=[p'(t)\, r(t)-2\,p(t)]/r^3(t)$. 

The family $\mathbf{X}_{a,b,c}$ of Jacobian elliptic K3 surfaces is obtained from the Seiberg-Witten curve $\mathbf{X}_{\scriptscriptstyle \mathrm{SW}}$
by the same base change using $f_{a,b,c}: \mathbb{CP}^1 \twoheadrightarrow \mathbb{CP}^1$ such that
\begin{equation}
 \hat{\tau}(\hat{E}_t) = \hat{\tau}(\hat{E}_u) \;\; \text{for all $t$ such that $u=f_{a,b,c}(t)$}\;.
\end{equation}
The elliptic fibration $\mathbf{X}_{a,b,c}$ is the fibered product of $\mathbf{X}_{\scriptscriptstyle \mathrm{SW}} \to \mathbb{CP}^1$ via $f_{a,b,c}$:
\begin{equation*}
\begin{array}{rclrc}
 \mathbf{X}_{a,b,c}  =  \mathbf{X}_{\scriptscriptstyle \mathrm{{SW}}} & \times_{\mathbb{CP}_u^1}   & \mathbb{CP}_t^1  & \longrightarrow             & \mathbb{CP}_t^1  \\
                                                   & \downarrow                 &                  &                             & \phantom{f_{a,b,c}}\downarrow f_{a,b,c}\\
                                                   & \mathbf{X}_{\scriptscriptstyle \mathrm{{SW}}} &                  & \longrightarrow             & \mathbb{CP}_u^1  
\end{array} 
\end{equation*}
The pull-back of the holomorphic two-form $\hat{\omega}_{\scriptscriptstyle \mathrm{SW}}$ is
\begin{equation}
\label{holomorphic_form}
 f_{a,b,c}^* \, \hat{\omega}_{\scriptscriptstyle \mathrm{SW}} =  r(t) \, \dfrac{\partial u}{\partial t}  \, \hat{\omega} \;.
\end{equation}
\noindent
The isogenies (\ref{transfoK3}) are obtained from Equation~(\ref{transfoSW}) by pull-back via the base change $u=f_{a,b,c}(t)$.
The rational quotient map (\ref{K3_quotient_map}) relating $\mathbf{X}_{a,b,c}$ to $\mathbf{Y}_{a,b,c}$ is obtained 
as the pull-back of the two-isogeny relating $\mathbf{X}_{\scriptscriptstyle \mathrm{{SW}}}$ to $\mathbf{Y}_{\scriptscriptstyle \mathrm{{SW}}}$.
\end{proof}

\section{Monodromies and Yukawa couplings}
\label{MonodromyPeriodsYukawa}
In this section, we will compute the monodromy matrices, the period integrals, and the Yukawa couplings associated with the families of Jacobian elliptic K3 surfaces $\mathbf{X}_{a,b,c}$
and $\mathbf{Y}_{a,b,c}$ and the Seiberg-Witten curves $\mathbf{X}_{\scriptscriptstyle \mathrm{SW}}$ and $\mathbf{Y}_{\scriptscriptstyle \mathrm{SW}}$.

\subsection{Monodromies and period integrals}
We fix a point $u_0 \in \mathbb{CP}^1$ that is a smooth fiber $E_{u_0}$ of the fibration $\mathbf{Y}_{\scriptscriptstyle \mathrm{SW}}$, and thus have $\dim H^1(E_{u_0})=2$. 
Since a base point is given in each fiber by the zero section, we can choose a symplectic basis $\lbrace A_u, B_u \rbrace$ of the homology $H_1(E_u)$ with respect to the intersection form
that changes analytically in $u$. That means that we have fixed a homological invariant, i.e.,  a locally constant sheaf
over the base whose generic stalk is isomorphic to $\mathbb{Z} \oplus \mathbb{Z}$. 

In addition, we will use $dX/Y$ as the analytical marking, i.e., the holomorphic one form on each regular fiber. 
We denote by $\Omega_{\scriptscriptstyle \mathrm{SW}}$ and $\Omega'_{\scriptscriptstyle \mathrm{SW}}$ the period integrals of $\omega_{\scriptscriptstyle \mathrm{SW}}$ 
over the A-cycle and the B-cycle respectively. It is possible to choose the homological invariant (cf. \cite{SW}) such that
the periods of $\omega_{\scriptscriptstyle \mathrm{SW}}$ over $A_u$ and $B_u$ become
\begin{equation}
\label{periods2}
\begin{split}
 2 \, \Omega_{\scriptscriptstyle \mathrm{SW}} \; du   & = \int_{A_u} \omega_{\scriptscriptstyle \mathrm{SW}} = \pi \, \vartheta^2_2(\tau) \; du\;, \\
 2 \, \Omega'_{\scriptscriptstyle \mathrm{SW}} \; du & = \int_{B_u} \omega_{\scriptscriptstyle \mathrm{SW}} = 2 \, \tau \; \Omega_{\scriptscriptstyle \mathrm{SW}} \; du\;,
 \end{split}
\end{equation}
where the relation between $u$ and $\tau$ was given in Equation~(\ref{u_coordinate2}). 

It follows that $\vec{\Omega}_{\scriptscriptstyle \mathrm{SW}} = (\Omega'_{\scriptscriptstyle \mathrm{SW}}, \Omega_{\scriptscriptstyle \mathrm{SW}})^t$
is a section of a holomorphic rank-two vector bundle, called the period bundle, over $\mathbb{CP}^1$ punctured
at $u =\pm 1$ and $u=\infty$. The vector bundle is equipped with a flat connection known as Gauss-Manin connection.
The connection has regular singularities at the base points of the singular fibers. The holonomy of the connection around the singular fibers 
determines a local system and a representation of the fundamental group of the punctured $\mathbb{CP}^1$ as follows:
let $\Gamma$ be the fundamental group of $\mathbb{CP}^1$ with the points $u=\pm1$ and $u=\infty$ removed.
Generators for the fundamental group $\Gamma$ are suitable simple loops $\alpha_u$ 
which run around the singular fibers at $u=1,-1,$ or $\infty$ counterclockwise, are based
at a point on the positive real line with $u_0>1$, and satisfy the relation $\alpha_{\infty} \alpha_{1} \alpha_{-1} =1$. 
We denote the monodromy matrices for the period integrals $\vec{\Omega}_{\scriptscriptstyle \mathrm{SW}}$ of $\mathbf{Y}_{\scriptscriptstyle \mathrm{SW}}$ 
around the loops $\alpha_u$ by $M^{\scriptscriptstyle \mathrm{SW}}_{u}$. These matrices generate the monodromy
representation $\Gamma \to \mathrm{SL}(2,\mathbb{Z})$ associated with the Jacobian elliptic surface 
$\mathbf{Y}_{\scriptscriptstyle \mathrm{SW}}$ together with the given homological and analytical marking.
It can be easily deduced from the transformation behavior of the Jacobi theta function in (\ref{periods2}) that
\begin{equation}
\begin{split}
 M^{\scriptscriptstyle \mathrm{SW}}_{-1}     = (TS) \, T^2 \, (TS)^{-1} &=\left( \begin{array}{rr} -1 & 2 \\  -2 & 3 \end{array} \right)  \,,\\
 M^{\scriptscriptstyle \mathrm{SW}}_{1}      = ST^2S^{-1} &= \left( \begin{array}{rr}   1 & 0 \\ -2 &  1 \end{array} \right) \,, \; \\
 M^{\scriptscriptstyle \mathrm{SW}}_{\infty} = - T^2 &=\left( \begin{array}{rr} -1 & -2 \\  0 & -1 \end{array} \right)  \,,\\
\end{split}
\end{equation}
such that the fundamental relation of $\Gamma$ becomes
\begin{equation}
 M^{\scriptscriptstyle \mathrm{SW}}_{\infty} \, M^{\scriptscriptstyle \mathrm{SW}}_{1} \, M^{\scriptscriptstyle \mathrm{SW}}_{-1} = \mathbb{I} \;,
\end{equation}
where $T$ and $S$ are the usual generators of the modular group
\begin{equation}
\label{generators}
 T = \left( \begin{array}{cc} 1 & 1 \\ 0 & 1 \end{array} \right)\;, \quad S=\left( \begin{array}{cc} 0 & -1 \\ 1 & 0 \end{array} \right) \;.
\end{equation}

For $\mathbf{Y}_{a,b,c}$, we use the relation (\ref{holomorphic_form2}) to compute the period integrals of $\omega$ over 
the A-cycle and B-cycle in the elliptic fiber $E_t$ with $u=f_{a,b,c}(t)$ from the period integrals for $\mathbf{Y}_{\scriptscriptstyle \mathrm{SW}}$ in Equation~(\ref{periods2}).  We then obtain
the following period integrals:
\begin{equation}
\label{K3periods2}
\begin{split}
  2 \, \Omega \, dt = \int_{A_t} \omega & =  \dfrac{2  \, \Omega_{\scriptscriptstyle \mathrm{SW}}}{r(t)} \;  dt\;, \\
  2 \, \Omega' \, dt = \int_{B_t} \omega & =  \dfrac{2  \, \Omega'_{\scriptscriptstyle \mathrm{SW}}}{r(t)} \; dt\;.
 \end{split}
\end{equation}
Let $\Gamma_{a,b,c}$ be the fundamental group of $\mathbb{CP}^1$ with the points $t=0,t_1, \dots,t_6$ and $t=\infty$
in the $t$-plane removed for which $\mathbf{Y}_{a,b,c}$ assumes a singular fiber. 
Generators for the fundamental group $\Gamma_{a,b,c}$ are suitable simple loops $\beta_t$ 
which run around the singular fibers at $t=0, t_1, \dots, t_6,$ or $t=\infty$ counterclockwise, are based
at a point on the positive real line with $u_0=f_{a,b,c}(t_0)$,
and satisfy the relation $\beta_{0} \beta_{t_1} \dots \beta_{t_4} \beta_{\infty} \beta_{t_5} \beta_{t_6}=1$. 
We denote the monodromy matrices for the period integrals $\vec{\Omega}= (\Omega', \Omega)^t$ of $\mathbf{Y}_{a,b,c}$ 
around the loops $\alpha_u$ by $M_{u}$. These matrices generate the monodromy
representation $\Gamma_{a,b,c} \to \mathrm{SL}(2,\mathbb{Z})$ associated with the Jacobian elliptic surface 
$\mathbf{Y}_{a,b,c}$ together with the given homological and analytical marking.

We now determine the monodromy  representation generated by the monodromy matrices for $\vec{\Omega}$ of $\mathbf{Y}_{a,b,c}$ along the particular choice of 
loops that satisfy
\begin{equation*}
\begin{split}
  f_{a,b,c \; *}\,[\beta_{t_{2i}}] = [\alpha_{-1}] \;, \quad f_{a,b,c \; *}\,[\beta_{t_{2i-1}}] = [\alpha_{1}]\;, \quad  f_{a,b,c \; *}\,[\beta_{\infty}]= f_{a,b,c \; *}\,[\beta_{0}]=[\alpha_\infty]\;,
\end{split}
\end{equation*}
where $[.]$ denotes homotopy classes in $\Gamma$ and $\Gamma_{a,b,c}$, respectively.
Since $u=-1$ for $p_{+}(t)=0$ which happens at $t=t_2, t_4, t_6$ and $u=1$ for $p_{-}(t)=0$ which happens at $t=t_1, t_3, t_5$,
we obtain for $1\le i \le 3$ the following monodromy matrices for $\vec{\Omega}$
\begin{equation}
\begin{split}
 M_{t_{2i}} &= M^{\scriptscriptstyle \mathrm{SW}}_{-1} =\left( \begin{array}{rr} -1 & 2 \\  -2 & 3 \end{array} \right)  \,,\\
 M_{t_{2i-1}}   &= M^{\scriptscriptstyle \mathrm{SW}}_{1}  = \left( \begin{array}{rr}   1 & 0 \\ -2 &  1 \end{array} \right) \,.
\end{split}
\end{equation}
Since $u=p(t)/r^2(t)\sim t$ for $t \gg0$, the monodromy matrix around infinity remains the same under the base change, i.e., 
\begin{equation}
 M_{\infty} = M^{\scriptscriptstyle \mathrm{SW}}_{\infty} =\left( \begin{array}{rr} -1 & -2 \\  0 & -1 \end{array} \right)  \,.
\end{equation}
Since $u=p(t)/r^2(t)\sim q^{-2}(t)$ for $r(t) \approx 0$, the monodromy matrix around $t=0$ is
\begin{equation} 
 M_{0} = \left( M^{\scriptscriptstyle \mathrm{SW}}_{\infty} \right)^{2} =\left( \begin{array}{rr} 1 & 4 \\  0 & 1 \end{array} \right) \;.
\end{equation}
The fundamental relation of $\Gamma_{a,b,c}$ becomes
\begin{equation}
\label{relations}
\begin{split}
 M_{0} \, M_{t_1} \, M_{t_2} \, M_{t_3} \, M_{t_4} \, M_{\infty} \, M_{t_5}\, M_{t_6} & = \mathbb{I} \;.
\end{split}
\end{equation}
%Additionally, we find the relations
%\begin{equation}
%\label{relations}
%\begin{split}
% M_{0} \, M_{t_1} \, M_{t_2} \, M_{t_3} \, M_{t_4} & = \mathbb{I} \;, \\
% M_{\infty} \, M_{t_5}\, M_{t_6} & = \mathbb{I} \;,\\
% M_{t_3} \, M_{t_4}\, M_{t_\infty} & = \mathbb{I} \;.
%\end{split}
%\end{equation}

The same construction can be carried out for the family of Jacobian elliptic K3 surfaces $\mathbf{X}_{a,b,c}$
which is obtained by a base change from the Seiberg-Witten curve $\mathbf{X}_{\scriptscriptstyle \mathrm{SW}}$.
It is possible to choose a homological invariant (cf. \cite{SW}) such that the periods of $\hat{\omega}_{\scriptscriptstyle \mathrm{SW}}$ over $A_u$ and $B_u$
of $\mathbf{X}_{\scriptscriptstyle \mathrm{SW}}$ become
\begin{equation}
\label{periods}
\begin{split}
 2 \, \hat{\Omega}_{\scriptscriptstyle \mathrm{SW}} \; du   & = \int_{A_u} \hat{\omega}_{\scriptscriptstyle \mathrm{SW}} = 2\pi \, \vartheta_2(\hat{\tau}) \, \vartheta_3(\hat{\tau}) \; du\;, \\
 2 \, \hat{\Omega}'_{\scriptscriptstyle \mathrm{SW}} \; du  & = \int_{B_u} \hat{\omega}_{\scriptscriptstyle \mathrm{SW}} = 2 \, \hat{\tau} \; \hat{\Omega}_{\scriptscriptstyle \mathrm{SW}} \; du\;,
 \end{split}
\end{equation}
where the relation between $u$ and $\hat{\tau}$ was given in Equation~(\ref{u_coordinate}). For the substitution $\tau=\hat{\tau}/2$ it follows from the
identity $2\, \vartheta_2(\tau) \, \vartheta_3(\tau) = \vartheta_2^2(\tau/2)$ that $\hat{\Omega}_{\scriptscriptstyle \mathrm{SW}}=\Omega_{\scriptscriptstyle \mathrm{SW}}$. 

By $\hat{M}^{\scriptscriptstyle \mathrm{SW}}_{u}$ we denote the monodromy matrix for $\hat{\vec{\Omega}}_{\scriptscriptstyle \mathrm{SW}}= (\hat{\Omega}'_{\scriptscriptstyle \mathrm{SW}}, \hat{\Omega}_{\scriptscriptstyle \mathrm{SW}})^t$ 
of $\mathbf{X}_{\scriptscriptstyle \mathrm{SW}}$ around the loop $\alpha_u\in \Gamma$. These matrices generate the monodromy
representation $\Gamma \to \mathrm{SL}(2,\mathbb{Z})$ associated with the Jacobian elliptic surface 
$\mathbf{X}_{\scriptscriptstyle \mathrm{SW}}$ together with the given homological and analytical marking.
It can be easily deduced from the transformation behavior of the Jacobi theta function in Equation (\ref{periods}) that
\begin{equation}
\begin{split}
 \hat{M}^{\scriptscriptstyle \mathrm{SW}}_{-1}     = (T^2S) T (T^2S)^{-1} & = \left( \begin{array}{rr} -1 & 4 \\  -1 & 3 \end{array} \right)  \,, \\
 \hat{M}^{\scriptscriptstyle \mathrm{SW}}_{1}      = STS^{-1} & = \left( \begin{array}{rr}   1 & 0 \\ -1 &  1 \end{array} \right) \,, \\
 \hat{M}^{\scriptscriptstyle \mathrm{SW}}_{\infty}  = - T^4 & = \left( \begin{array}{rr} -1 & -4 \\  0 & -1 \end{array} \right)  \,,
\end{split}
\end{equation}
such that the fundamental relation of $\Gamma$ becomes
\begin{equation}
 \hat{M}^{\scriptscriptstyle \mathrm{SW}}_{\infty} \, \hat{M}^{\scriptscriptstyle \mathrm{SW}}_{1} \, \hat{M}^{\scriptscriptstyle \mathrm{SW}}_{-1} = \mathbb{I} \;.
\end{equation}
For $\mathbf{X}_{a,b,c}$, we compute the period integrals of $\hat{\omega}$ over 
the A-cycle and B-cycle in the elliptic fiber $E_t$ with $u=f_{a,b,c}(t)$ from the period integrals for $\mathbf{X}_{\scriptscriptstyle \mathrm{SW}}$ in Equation~(\ref{periods}).  We obtain
the following period integrals:
\begin{equation}
\label{K3periods}
\begin{split}
 2 \, \hat{\Omega} \, dt =  \int_{A_t} \hat{\omega} & =  \dfrac{2  \, \hat{\Omega}_{\scriptscriptstyle \mathrm{SW}}}{r(t)} \;  dt\;, \\
 2 \, \hat{\Omega}' \, dt =  \int_{B_t} \hat{\omega} & =  \dfrac{2  \, \hat{\Omega}'_{\scriptscriptstyle \mathrm{SW}}}{r(t)} \; dt\;.
 \end{split}
\end{equation}
By $\hat{M}_{t}$ we denote the monodromy matrix for $\hat{\vec{\Omega}}= (\hat{\Omega}', \hat{\Omega})^t$ 
of $\mathbf{X}_{a,b,c}$ around the loop $\beta_t\in \Gamma_{a,b,c}$. These matrices generate the monodromy
representation $\Gamma_{a,b,c} \to \mathrm{SL}(2,\mathbb{Z})$ associated with the Jacobian elliptic surface 
$\mathbf{X}_{a,b,c}$ together with the given homological and analytical marking.
With the same choice of generators $\beta_0, \beta_{t_1}, \dots,\beta_{t_6},\beta_\infty$ for the fundamental group 
$\Gamma_{a,b,c}$ as before, we obtain for $1\le i \le 3$ that
\begin{equation}
\begin{array}{lclclcl}
 \hat{M}_{t_{2i}} & = & \hat{M}^{\scriptscriptstyle \mathrm{SW}}_{-1}\;,  &\qquad&
 \hat{M}_{t_{2i-1}}   & = & \hat{M}^{\scriptscriptstyle \mathrm{SW}}_{1} \;,   \\
 \hat{M}_{\infty}   & = & \hat{M}^{\scriptscriptstyle \mathrm{SW}}_{\infty} \;, &&
 \hat{M}_{0}        & = & \left( \hat{M}^{\scriptscriptstyle \mathrm{SW}}_{\infty} \right)^{2} \;.
\end{array}
\end{equation}
Thus, we have proved the following lemma:
\begin{lemma}
\label{parts3b}
The ratio of the period integrals of the holomorphic two-forms $\omega_{\scriptscriptstyle \mathrm{SW}}$ and $\omega$ 
over the A-cycle and B-cycle in the elliptic fiber $E_u$ and $E_t$ with $u=f_{a,b,c}(t)$ 
of $\mathbf{X}_{\scriptscriptstyle \mathrm{SW}}$ and $\mathbf{X}_{a,b,c}$, respectively,
is given by
\begin{equation}
\label{K3periods_ratio}
  \dfrac{\int_{B_u} \hat{\omega}_{\scriptscriptstyle \mathrm{SW}}}{ \int_{B_t} \hat{\omega} }  =  \dfrac{\int_{A_u} \hat{\omega}_{\scriptscriptstyle \mathrm{SW}}}{ \int_{A_t} \hat{\omega} }  =  r(t) \;  \frac{\partial u}{\partial t} 
\end{equation}  
\end{lemma}

\subsection{The Yukawa couplings}
We first introduce the notion of a special K\"ahler structure.
\begin{definition}
On the K\"ahler manifold $M$ with K\"ahler form $\eta$, a special K\"ahler structure
is a real flat torsion-free symplectic connection $\nabla$ satisfying $d^\nabla I =0$ where $I$ is the complex structure on $M$.
\end{definition}
\noindent
On a special K\"ahler manifold $M$, there is a holomorphic cubic form $\Xi$ which encodes the extend to which $\nabla$ fails to preserve the complex structure
\cite[Equation~(1.26)]{Freed}, i.e.,
\begin{equation}
 \Xi = - \eta\Big( \pi^{(1,0)}, \nabla \pi^{(1,0)}\Big) \in H^0(M, \mathrm{Sym}^3T^*M) \;,
\end{equation}
where $\pi^{(1,0)}$ is the projection onto the $(1,0)$-part of the complexified tangent bundle.
In physics, the holomorphic cubic form is called the Yukawa coupling. One class of examples of special 
K\"ahler manifolds comes from algebraic integrable systems \cite{Freed}. In the case of a Jacobian elliptic fibration
over $\mathbb{CP}^1$ equipped with a holomorphic symplectic two-form $\omega$, the form $\eta$ obtained from 
integrating $\omega \wedge \overline{\omega}$ over the regular elliptic fiber
defines a special K\"ahler structure on the punctured base curve $\mathbb{CP}^1$.

\noindent
We have the following lemma:
\begin{lemma}
\label{parts3c}
The holomorphic cubic 
forms $\hat{\Xi}_{\scriptscriptstyle \mathrm{{SW}}}$ on $\mathbf{X}_{\scriptscriptstyle \mathrm{SW}}$ and $\hat{\Xi}$ on $\mathbf{X}$ are related by
\begin{equation*}
 f_{a,b,c}^* \Big(\hat{\Xi}_{\scriptscriptstyle \mathrm{SW}}\Big) = r^2(t) \, \left( \dfrac{\partial u}{\partial t} \right)^2 \;  \hat{\Xi} \;.
\end{equation*}
\end{lemma}
\begin{proof}
The Seiberg-Witten curve $\mathbf{Y}_{\scriptscriptstyle \mathrm{SW}}$ in Equation~(\ref{WE3}) equipped with $\omega_{\scriptscriptstyle \mathrm{SW}}$ defines a special K\"ahler structure
on $\mathbb{CP}^1 - \{ -1, 1, \infty\}$ with
\begin{equation}
 \eta_{\scriptscriptstyle \mathrm{SW}} = 4 \, \im\,\tau \, \left| \Omega \right|^2 \; du \wedge d\bar{u} \;.
\end{equation}
We compute the holomorphic cubic form $\Xi_{\scriptscriptstyle \mathrm{SW}}$ in the affine coordinate $u$ using \cite[Equation~(1.28)]{Freed} and obtain
\begin{equation}
 \Xi_{\scriptscriptstyle \mathrm{SW}} = \frac{1}{4} \; \Omega^2 \, \dfrac{\partial\tau}{\partial u} \; (du)^{\otimes 3} \;,
\end{equation}
where the period $\Omega$ was computed in Equation~(\ref{periods2}). Similarly, for the special K\"ahler manifold $\mathbf{Y}_{a,b,c}$
equipped with $\omega$ we obtain the holomorphic cubic form
\begin{equation}
\label{Xi}
 \Xi = \frac{1}{4} \; \dfrac{\Omega^2}{r^2(t)} \, \dfrac{\partial\tau}{\partial t} \; (dt)^{\otimes 3} \;.
\end{equation}
Comparing Equation~(\ref{Xi}) to the pull-back of $\Xi_{\scriptscriptstyle \mathrm{SW}}$ via the base change $f_{a,b,c}$,
we obtain the following relation between the Yukawa coupling derived from the Seiberg-Witten curve $\mathbf{Y}_{\scriptscriptstyle \mathrm{SW}}$
and the family of K3 surfaces $\mathbf{Y}_{a,b,c}$
\begin{equation}
 f_{a,b,c}^* \Big(\Xi_{\scriptscriptstyle \mathrm{SW}}\Big) = r^2(t) \, \left( \dfrac{\partial u}{\partial t} \right)^2 \;  \Xi \;.
\end{equation}
The same construction can be carried out for the family of Jacobian elliptic K3 surfaces $\mathbf{X}_{a,b,c}$
which is obtained by a base change from the Seiberg-Witten curve $\mathbf{X}_{\scriptscriptstyle \mathrm{SW}}$.
For $\mathbf{X}_{\scriptscriptstyle \mathrm{SW}}$, we obtain the following holomorphic cubic form $\hat{\Xi}_{\scriptscriptstyle \mathrm{SW}}$ in the affine coordinate $u$
\begin{equation}
 \hat{\Xi}_{\scriptscriptstyle \mathrm{SW}} = \frac{1}{4} \; \hat{\Omega}^2 \, \dfrac{\partial\hat{\tau}}{\partial u} \; (du)^{\otimes 3} = \frac{1}{2} \; \Xi_{\scriptscriptstyle \mathrm{SW}}\;.
\end{equation}
For $\mathbf{X}_{a,b,c}$, we obtain the holomorphic cubic form
\begin{equation}
 \hat{\Xi} = \frac{1}{4} \; \dfrac{\hat{\Omega}^2}{r^2(t)} \, \dfrac{\partial\hat{\tau}}{\partial t} \; (dt)^{\otimes 3} = \frac{1}{2} \; \Xi\;,
\end{equation}
such that
\begin{equation}
 f_{a,b,c}^* \Big(\hat{\Xi}_{\scriptscriptstyle \mathrm{SW}}\Big) = r^2(t) \, \left( \dfrac{\partial u}{\partial t} \right)^2 \;  \hat{\Xi} \;.
\end{equation}
This completes the proof of the lemma.
\end{proof}

\section{The $(16,6)$-configuration}
\label{kummer_surface}
In this section, we will determine the parameters $a, b, c$ in the 3-parameter family of Jacobian elliptic K3 surfaces $\mathbf{Y}_{a,b,c}$
in terms of the roots of the sextic curve defining $\mathbf{C}$ such that $\mathbf{Y}_{a,b,c} = \mathrm{Kum}(J(\mathbf{C}))$.

\subsection{Keum's results}
In \cite[Part 2]{Keum}, Keum describes the elliptic pencil on a Kummer quartic $\hat{\mathrm{K}}$ in the $(16,6)$-configuration introduced in
Section \ref{Kummer_surfaces}. We briefly review Keum's construction:
first, he chooses a node $p_{\alpha}$. He considers the projection  $\pi_\alpha: \hat{\mathrm{K}} \to \mathbb{CP}^2$ from the double point 
$p_\alpha$ of $\mathrm{K}$ to $\mathbb{CP}^2$ and the canonical resolution of the double cover. 
The six lines in the branch locus are the images of six tropes which intersect $p_{\alpha}$. The intersection points of these six lines are the image of 
$15$ nodes and $p_{\alpha}$. Let $q_0, q_{ij}$ be the projections of the nodes $p_0, p_{ij}$ by $\pi_\alpha$.

Keum then chooses four points $q_\beta, q_\gamma, q_\delta, q_\epsilon$ such that each of the pairs $\{q_\beta, q_\gamma\}$ and 
$\{q_\delta, q_\epsilon\}$ lies on one of the six lines, and none of the pairs $\{q_\beta, q_\delta\}$, $\{q_\beta, q_\epsilon\}$,
$\{q_\gamma, q_\delta\}$, $\{q_\gamma, q_\epsilon\}$ lies on any of the six lines. Then, he considers the pencil
$|Q - q_\beta - q_\gamma - q_\delta - q_\epsilon|$  of conics passing through the four points $q_\beta, q_\gamma, q_\delta, q_\epsilon$. The 
pull-back 
$$\pi_\alpha^* |Q - q_\beta - q_\gamma - q_\delta - q_\epsilon|$$ 
is an elliptic pencil
on $\hat{\mathrm{K}}$. In \cite[Theorem 2]{Keum}  Keum proves the following theorem:
\begin{theorem}[Keum]
\label{theorem_Keum}
If for the Kummer surface $\hat{\mathrm{K}}$ introduced in Section \ref{Kummer_surfaces}, the genus-two curve $\mathbf{C}$ is generic
then there exists an elliptic fibration on it with six singular fibers of Kodaira-type $I_2$, one singular fiber 
of Kodaira-type $I_4$, one singular fiber of Kodaira-type $I_2^*$, and Mordell-Weil group $(\mathbb{Z}_2)^2$.
\end{theorem}
\noindent
We have the following result:
\begin{corollary}
\label{parts4a}
For generic values of the parameters $a, b, c$,
the K3 surface $\mathbf{Y}_{a,b,c}$ in Equation~(\ref{WE}) is the Kummer surface of the Jacobian variety of generic curve of genus two.
\end{corollary}
\begin{proof}
Comparing the result of Theorem \ref{theorem_Keum} with the Jacobian elliptic fibration on $\mathbf{Y}_{a,b,c}$ in Equation~(\ref{WE})
proves the statement.
\end{proof}

\noindent
\subsection{The pencil of conics}
To describe the pencil of conics explicitly, let us first re-write the elliptic fibration (\ref{WE}) using the affine coordinate
$r=t-c$. The polynomial $p(t)$ becomes
\begin{equation*}
 p(r) = 4 \, r^3 + 12 \, C \, r^2 - 3 \, A \, r - B \;,
\end{equation*}
with
\begin{equation}
\label{parameter_abc}
 a= A + 4 \, C^2 \;, \quad b = B - 3\, A \, C - 8 \, C^3 \;, \quad c = C \;.
\end{equation}
For $p_\pm(r) = p(r) \mp r^2$, we denote the roots of $p_+(r)=0$ and $p_-(r)=0$ by $r_2, r_4, r_6$
and $r_1, r_3, r_5$ respectively. Hence we have
\begin{equation*}
\begin{split}
\hat{p}_{+}(r)   & =  p(r) - r^2 = 4 \, \prod_{i=1}^3 (r-r_{2i}) \;,\\
\hat{p}_{-}(r)   & =  p(r) + r^2 = 4 \, \prod_{i=1}^3 (r-r_{2i-1}) \;.
\end{split}
\end{equation*}
The equation
\begin{eqnarray*}
 p(r) & = & 2 \left( \prod_{i=1}^3 (r-r_{2i-1}) + \prod_{i=1}^3 (r-r_{2i}) \right) = 4 \, r^3 + 12 \, C \, r^2 - 3 \, A \, r - B 
\end{eqnarray*}
implies the following relation between the roots $r_i$ and the parameters $A, B, C$:
\begin{equation}
\label{parameter_ABC}
\begin{split}
 - 3 \, A & = 2 \, \left( r_1 r_3 + r_1 r_5 + r_3 r_5 + r_2 r_4 + r_2 r_6  + r_4 r_6 \right) \;,\\
        B & = 2 \, \left( r_1 r_3 r_5 + r_2 r_4 r_6 \right) \;,\\
 - 6 \, C & = r_1 + r_2 + r_3 + r_4 + r_5 + r_6 \;.
\end{split}
\end{equation}
The equation
\begin{eqnarray*}
 r^2 & = & 2 \left( \prod_{i=1}^3 (r-r_{2i-1}) - \prod_{i=1}^3 (r-r_{2i}) \right)  
\end{eqnarray*}
implies the following constraints on the roots:
\begin{eqnarray}
\label{constraint1}
 0  & = & r_1 r_3 r_5 - r_2 r_4 r_6 \;,\\
\label{constraint2}
 0  & = & r_2 r_4 + r_2 r_6 + r_4 r_6 - r_1 r_3 - r_1 r_5 - r_3 r_5 \;,\\
\label{constraint3}
 1  & = & 2\, \left(r_2 + r_4 + r_6 - r_1 - r_3 -r_5\right) \;.
\end{eqnarray}
We apply Keum's result for $p_{\alpha}=p_0$ and consider the elliptic pencil of conics $|Q - q_{14} - q_{15} - q_{23} - q_{26}|$
through the points $\{q_{14}, q_{15}, q_{23}, q_{26}\}$. The pull-back 
$$\pi_0^*(Q - q_{14} - q_{15} - q_{23} - q_{26})$$
is an elliptic pencil  with the same configuration of singular fibers and Mordell-Weil group as the Weierstrass fibration in Equation~(\ref{WE}).
In fact, from the proof of \cite[Theorem 2]{Keum} it also follows that the zero section and the two-torsion sections are $T_3, T_4, T_5, T_6$.

\noindent
We write the conic in $\mathbb{CP}^2$ in the form
\begin{equation*}
 Q = A_1 z_1^2 + A_2 z_2^2 + A_3 z_3^2 + 2 B_3 z_1 z_2 + 2 B_2 z_1 z_3 + 2 B_1 z_2 z_3 =0 \;,
\end{equation*}
and set
\begin{equation*}
 R = \left( \begin{array}{ccc}A_1 & B_3 & B_2 \\ B_3 & A_2 & B_1 \\ B_2 & B_1 & A_3 \end{array} \right) \;.
\end{equation*}
The conic degenerates if and only it decomposes into a product of two lines, i.e., $\det R=0$. 
This happens at the points where the elliptic pencil develops a singular fiber.
The coordinates of the points $q_{ij}$ are $ q_{ij} = [1: \theta_i + \theta_j : \theta_i \, \theta_j]$.
It is easy to show that for the conic section that runs through $q_{14}, q_{15}, q_{23}, q_{26}$
we can choose $\frac{B_1}{A_2}$ as a free parameter. 

Using the results of Keum \cite{Keum}, it follows from the description of the elliptic pencil in the $(16,6)$-configuration that at 
the singular fiber of Kodaira type $I_2^*$ the conic $Q$ degenerates to the sum of two lines joining $q_{14}$ and $q_{15}$, $q_{23}$ 
and $q_{26}$ respectively. On the other hand, the Weierstrass fibration
(\ref{WE}) on $\mathbf{Y}_{a,b,c}$ assumes a singular fiber of Kodaira type $I_2^*$ at $r=\infty$.
Hence, we have shown that $Q_{\infty}=\pi_0(T_1) \cdot \pi_0(T _2)$. At the singular fiber of Kodaira type $I_4$, the conic passes through
the six points $q_{14}, q_{15}, q_{23}, q_{26}$ and $q_{45}, q_{36}$. Comparison with the Weierstrass fibration
(\ref{WE}) shows that the singular fiber of Kodaira type $I_4$ must be located at $r=0$.
Therefore, we have determined that the variable $r$ is $Q_0/Q_{\infty}$ up to a scale factor, since we have determined the correct numerator and denominator. 
The denominator blows up when $Q_\infty=0$ which happens when $r=\infty$, the numerator is zero when $r=0$. Thus, we have
\begin{equation}
\label{r_variable}
  r  =   \alpha \; \frac{Q_{0}}{Q_\infty} \;. 
\end{equation}
Similarly, we can determine for the remaining six singular fibers in what way the conic $Q$ degenerates and at what value for the variable $r$
in the elliptic fibration (\ref{WE}) this will happen. We conclude:

\noindent
At $r_2$, the conic $Q=Q_{r_2}$ is the product of two lines joining $q_{14}$ and $q_{23}$, $q_{15}$ and $q_{26}$.

\noindent
At $r_1$, the conic $Q=Q_{r_1}$ is the product of two lines joining $q_{14}$ and $q_{26}$, $q_{15}$ and $q_{23}$.

\noindent
At $r_4$, the conic $Q=Q_{r_4}$ contains the points $q_{14}, q_{15}, q_{23}, q_{26}$ and $q_{34}$. 

\noindent
At $r_6$, the conic $Q=Q_{r_6}$ contains the points $q_{14}, q_{15}, q_{23}, q_{26}$ and $q_{56}$.  

\noindent
At $r_3$, the conic $Q=Q_{r_3}$ contains the points $q_{14}, q_{15}, q_{23}, q_{26}$ and $q_{35}$.  

\noindent
At $r_5$, the conic $Q=Q_{r_5}$ contains the points $q_{14}, q_{15}, q_{23}, q_{26}$ and $q_{46}$.  

\noindent 
This enables us to compute the roots $r_i$, i.e., the location of the singular fibers of Kodaira-type $I_2$, up to a scale factor:
$r-r_i$ is $Q_{r_i}/Q_{\infty}$ up to a scale factor, since we have determined the correct numerator and denominator. 
It then follows from Equation~(\ref{r_variable}) that
\begin{eqnarray*}
 (r - r_i) & = &  \alpha \, \frac{C_{r_i}}{C_\infty}  \quad \Rightarrow \quad r_i = \alpha \; \frac{C_{0}-C_{r_i}}{C_\infty} \;.  
\end{eqnarray*}
We computed the conics at the points $r_i$, solved for $r_i$, and obtained
\begin{equation*}
\begin{array}{rcl}
 r_1 & = &  \alpha \, (\theta_3 - \theta_5) (\theta_4 - \theta_6)\;,\\
 r_2 & = &  \alpha \, (\theta_3 - \theta_4) (\theta_5 - \theta_6)\;,\\ 
 r_3 & = &  \alpha \, \dfrac{(\theta_1 - \theta_2) (\theta_3 - \theta_5) (\theta_4 - \theta_3) (\theta_5-\theta_6)}
                      {(\theta_1 - \theta_5)(\theta_2-\theta_3)} \;,\\
 r_4 & = &  \alpha \, \dfrac{(\theta_1 - \theta_2) (\theta_3 - \theta_4) (\theta_4 - \theta_6) (\theta_5-\theta_3)}
                      {(\theta_1 - \theta_4)(\theta_2-\theta_3)} \;,\\ 
 r_5 & = &  \alpha \, \dfrac{(\theta_1 - \theta_2) (\theta_3 - \theta_4) (\theta_4 - \theta_6) (\theta_5-\theta_6)}
                      {(\theta_1 - \theta_4)(\theta_2-\theta_6)} \;,\\
 r_6 & = &  \alpha \, \dfrac{(\theta_1 - \theta_2) (\theta_3 - \theta_5) (\theta_4 - \theta_6) (\theta_5-\theta_6)}
                      {(\theta_1 - \theta_5)(\theta_2-\theta_6)} \;.
\end{array}
\end{equation*}
We checked that these roots satisfy the constraints Equation~(\ref{constraint1}) and Equation~(\ref{constraint2}).
We used Equation~(\ref{constraint3}) to determine $\alpha$ and obtained
\begin{equation}
\label{parameter}
\begin{split}
 r_1 & =   \frac { \left( \theta_{{1}}-\theta_{{5}} \right)  \left( \theta_{{2}}-\theta_{{6}} \right)  \left( \theta_{{3}}-\theta_{{2}} \right)
\left( \theta_{{4}}-\theta_{{1}} \right)  \left( \theta_{{5}}-\theta_{{3}} \right)  \left( \theta_{{4}}-\theta_{{6}} \right) }
{2 \, (\theta_1 - \theta_3)(\theta_2-\theta_4)(\theta_3-\theta_6)(\theta_4-\theta_5)(\theta_5-\theta_2)(\theta_6-\theta_1)}\;,\\
 r_2 & =  \frac { \left( \theta_{{1}}-\theta_{{5}} \right)  \left( \theta_{{2}}-\theta_{{6}} \right)  \left( \theta_{{3}}-\theta_{{2}} \right)
\left( \theta_{{4}}-\theta_{{1}} \right)  \left( \theta_{{5}}-\theta_{{6}} \right)  \left( \theta_{{4}}-\theta_{{3}} \right) }
{2 \, (\theta_1 - \theta_3)(\theta_2-\theta_4)(\theta_3-\theta_6)(\theta_4-\theta_5)(\theta_5-\theta_2)(\theta_6-\theta_1)}\;,\\
 r_3 & =  \frac { \left( \theta_{{1}}-\theta_{{2}} \right)  \left(\theta_{{2}}-\theta_{{6}} \right)  \left( \theta_{{3}}-\theta_{{4}} \right)
 \left( \theta_{{4}}-\theta_{{1}} \right)  \left( \theta_{{5}}- \theta_{{3}} \right)  \left( \theta_{{5}}-\theta_{{6}} \right) }
{2 \, (\theta_1 - \theta_3)(\theta_2-\theta_4)(\theta_3-\theta_6)(\theta_4-\theta_5)(\theta_5-\theta_2)(\theta_6-\theta_1)}\;,\\
 r_4 & =  \frac { \left( \theta_{{1}} -\theta_{{2}} \right)  \left( \theta_{{2}}-\theta_{{6}} \right)  \left( \theta_{{3}}-\theta_{{5}} \right)
 \left( \theta_{{4}}-\theta_{{3}} \right)  \left( \theta_{{5}}-\theta_{{1}} \right)  \left( \theta_{{4}}-\theta_{{6}} \right) }
{2 \, (\theta_1 - \theta_3)(\theta_2-\theta_4)(\theta_3-\theta_6)(\theta_4-\theta_5)(\theta_5-\theta_2)(\theta_6-\theta_1)}\;,\\
 r_5 & =  \frac { \left( \theta_{{1}}-\theta_{{2}} \right)  \left( \theta_{{2}}-\theta_{{3}} \right)  \left( \theta_{{3}}-\theta_{{4}} \right) 
 \left( \theta_{{4}}-\theta_{{6}} \right)  \left( \theta_{{5}}-\theta_{{1}} \right)  \left( \theta_{{5}}-\theta_{{6}} \right) }
{2 \, (\theta_1 - \theta_3)(\theta_2-\theta_4)(\theta_3-\theta_6)(\theta_4-\theta_5)(\theta_5-\theta_2)(\theta_6-\theta_1)}\;,\\
 r_6 & =   \frac { \left( \theta_{{1}}-\theta_{{2}} \right)  \left( \theta_{{2}}-\theta_{{3}} \right)  \left( \theta_{{3}}-\theta_{{5}} \right) 
 \left( \theta_{{4}}-\theta_{{1}} \right)  \left( \theta_{{5}}-\theta_{{6}} \right)  \left( \theta_{{4}}-\theta_{{3}} \right) }
 {2 \, (\theta_1 - \theta_3)(\theta_2-\theta_4)(\theta_3-\theta_6)(\theta_4-\theta_5)(\theta_5-\theta_2)(\theta_6-\theta_1)}\;.
\end{split}
\end{equation}
We can substitute the roots (\ref{parameter}) into Equations (\ref{parameter_ABC}) to obtain $A, B, C$. In turn, $A, B, C$ can be substituted into Equations 
(\ref{parameter_abc}) to obtain $a, b, c$.

\noindent 
Thus, we have proved the following lemma:
\begin{lemma}
\label{parts4b}
For generic values of the parameters $a, b, c$, the K3 surface $\mathbf{Y}_{a,b,c}$ is 
the Kummer surface of the Jacobian variety of a generic curve $\mathbf{C}$ of genus two.
Equations~(\ref{parameter_abc}), (\ref{parameter_ABC}), and (\ref{parameter}) express the complex 
parameters $a, b, c$ defining $\mathbf{Y}_{a,b,c}$ in terms of the roots of the sextic curve (\ref{curve_2}) defining $\mathbf{C}$.
\end{lemma}

\section{Proof of Theorem \ref{maintheorem}}
\label{proof}
\noindent
Statement $(1)$ of Theorem \ref{maintheorem} is Lemma \ref{parts1}.

\noindent
Statement $(2)$ of Theorem \ref{maintheorem}  is Lemma \ref{parts2}.

\noindent
Statement $(3)$ of Theorem \ref{maintheorem}  is Lemma \ref{parts3a}, \ref{parts3b}, \ref{parts3c}.

\noindent
Statement $(4)$ of Theorem \ref{maintheorem}  is Corollary \ref{parts4a}  and Lemma \ref{parts4b}.

\section{Conlcusion}
We showed that for generic values of the parameters $a, b, c$ the K3 surfaces $\mathbf{Y}_{a,b,c}$ are
the Kummer surfaces of the Jacobian variety of the curve $\mathbf{C}$ of genus two. We proved this by finding two divisors, one at $t=0$ and the other for the fiber at $t=\infty$ to write 
down an actual elliptic fibration on the Kummer surface associated with a generic curve of genus two. The elliptic fibration matches the one on $\mathbf{Y}_{a,b,c}$ 
and enabled us to express the parameters $a, b, c$ in terms of the roots of the sextic curve defining $\mathbf{C}$.
Thus, $\mathbf{X}_{a,b,c}$ is a  K3 surface with a Nikulin involution and is a rational double-cover of $\mathbf{Y}_{a,b,c}$.
The rational quotient map induces a Hodge isometry $\mathrm{T}_{\mathbf{X}}(2) \cong \mathrm{T}_{\mathbf{Y}}$ over $\mathbb{Q}$.
This is in accordance with the fact that if there is a Hodge isometry between $\mathrm{T}_{\mathbf{X}}(2)\otimes \mathbb{Q}$ and $\mathrm{T}_{\mathbf{Y}}\otimes \mathbb{Q}$ then their determinants differ by a square in $\mathbb{Q}$ -- which $2^8=2^5\cdot 2^3$ and $2^6$ do. 

\section*{Acknowledgments}
\noindent
I would like to thank David Morrison for many helpful 
discussions and a lot of encouragement. I would also like to thank the referee
for many helpful comments.

\appendix

\section{The discriminant group}
Let $(\mathrm{L},Q)$ be an even lattice over $\mathbb{Z}$ with a quadratic form $Q$ and the dual lattice $\mathrm{L}^\vee$.
The discriminant group is by definition $D(\mathrm{L})=\mathrm{L}^\vee/\mathrm{L}$.
The natural projection is denoted by $\Psi: \mathrm{L}^{\vee} \to D(\mathrm{L})$. A non-degenerate quadratic form $q_{\mathrm{L}}$ on $D(\mathrm{L})$ with 
values in $\mathbb{Z}/2 \mod{2}$ is given by $q_{\mathrm{L}}(x) = Q(x') \mod{2}$ where $x' \in \mathrm{L}^\vee$ such that $\Psi(x')=x$.

There is a bijection between the set of isotopic subgroups of the discriminant group $(D(\mathrm{L}),q_{\mathrm{L}})$ and the set
of overlattices of $\mathrm{L}$, i.e., the integral sub-lattices of $\mathrm{L}^\vee$ containing $\mathrm{L}$.
\begin{theorem}[Nikulin \cite{Nikulin2}]
If $\mathrm{V} \subset D(\mathrm{L})$ is a subgroup isotopic with respect
to $q_{\mathrm{L}}$, then $\mathrm{M}=\Psi^{-1}(\mathrm{V})$ is an even overlattice of $\mathrm{L}$, and the discriminant form
of $\mathrm{M}$ is isomorphic to
\begin{equation*}
  \Big( D(\mathrm{M}), q_{\mathrm{M}} \Big) =\Big( D(\mathrm{L}), q_{\mathrm{L}} \Big)\Big|_{\mathrm{V}^\perp/\mathrm{V}} \;.
\end{equation*}
The correspondence $\mathrm{V} \to \mathrm{M}$ is a bijection.
\end{theorem}

\noindent
For the lattice of $\mathrm{D}_m$, we denote a basis of $\mathrm{L}^\vee$
by $\lbrace d_1^*, \dots , d_m^* \rbrace$, which is dual to a basis of simple roots 
$\lbrace d_1, \dots , d_m \rbrace$. By $\bar{d}_i^*$ we denote their images under 
$\Psi_{\mathrm{L}}$. We use the analogous notation for the lattice of $\mathrm{A}_m$.

\subsection{The case $\mathcal{W}^{\mathrm{root}}= \mathrm{D}_{16}$}
\label{discriminant1}
In the case that $\mathcal{W}^{\mathrm{root}}= \mathrm{D}_{16}$ and $\mathrm{MW}(\mathbf{X})=\mathbb{Z}_2$ for the K3 surface $\mathbf{X}$, we have
a discriminant group
\begin{equation*}
 \Big( D(\mathcal{W}^{\mathrm{root}}), q_{\mathcal{W}^{\mathrm{root}}} \Big) = \left( \mathbb{Z}/2 \oplus \mathbb{Z}/2  
, \left[ \begin{array}{cc} 4 & \frac{1}{2} \\ \frac{1}{2} & 1 \end{array} \right] \right) \;.
\end{equation*}
Generators are $\langle \bar{d}_1^* \rangle \oplus \langle \bar{d}_{16}^* \rangle \cong  \mathbb{Z}/2 \oplus \mathbb{Z}/2$.
We have $q(\bar{d}_1^*)=4 \equiv 0 (2)$ and $\langle \bar{d}_1^* \rangle^{\perp} \cong \langle \bar{d}_1^* \rangle$.
It follows that for the isotopic subgroup $\mathrm{V}=\langle \bar{d}_1^* \rangle$ we have
\begin{equation*}
 \Big( D(\mathcal{W}^{\mathrm{root}}), q_{\mathcal{W}^{\mathrm{root}}} \Big)\Big|_{\mathrm{V}^\perp/\mathrm{V}} = \Big( 0, \left[0 \right] \Big) \;.
\end{equation*}
$\Psi^{-1}(\mathrm{V})$ is an even overlattice of $\mathcal{W}^{\mathrm{root}}$ with $\Psi^{-1}(\mathrm{V})/\mathcal{W}^{\mathrm{root}}\cong \mathbb{Z}_2$.
Hence, $\mathcal{W}$ is $\Psi^{-1}(\bar{d}_1^*)=\mathrm{D}^+_{16}$.
By \cite[Prop. 1.6.1]{Nikulin2}, the transcendental lattice $\mathrm{T}_{\mathbf{X}}$ satisfies
\begin{equation*}
 \Big( D(\mathrm{T}_{\mathbf{X}}), q_{\mathrm{T}_{\mathbf{X}}} \Big) \cong \Big( D(\mathcal{W}), -q_{\mathcal{W}} \Big) \;.
\end{equation*}
Thus, the transcendental lattice has signature $(2,2)$ and discriminant group $(0)$.
By \cite[Corollary 2.9(iii)]{Morrison84}, a lattice $\mathrm{T}$ of signature $(2,2)$ occurs as transcendental 
lattice of a K3 surface if and only if $\mathrm{T} \cong \mathrm{H} \oplus \mathrm{T}'$ where $\mathrm{T}'$ is a lattice of signature $(1,1)$.
Thus, we have $\mathrm{T}_{\mathbf{X}} =  \mathrm{H}^2$.

\subsection{The case $\mathcal{W}^{\mathrm{root}}= \mathrm{D}_8 \oplus \mathrm{A}_7$}
\label{discriminant2}
In the case $\mathcal{W}^{\mathrm{root}}= \mathrm{D}_8 \oplus \mathrm{A}_7$ and $\mathrm{MW}(\mathbf{X})=\mathbb{Z}_2$ for $\mathbf{X}$,
we have
\begin{equation*}
 \Big( D(\mathcal{W}^{\mathrm{root}}), q_{\mathcal{W}^{\mathrm{root}}} \Big) = \left( \Big(\mathbb{Z}/2 \oplus \mathbb{Z}/2 \Big) 
\oplus \mathbb{Z}/8, \left[ \begin{array}{cc} 2 & \frac{1}{2} \\ \frac{1}{2} & 1 \end{array} \right] \oplus \left[\frac{1}{8} \right] \right) \;.
\end{equation*}
Generators are $\langle \bar{a}^*_7 \rangle \cong \mathbb{Z}/8$ and 
$\langle \bar{d}_1^* \rangle \oplus \langle \bar{d}_8^* \rangle  \cong \mathbb{Z}/2 \oplus \mathbb{Z}/2$.
We have $q(\bar{d}_1^*)=2 \equiv 0 (2)$ and $\langle \bar{d}_1^* \rangle^{\perp} \cong \langle \bar{d}_1^* \rangle \oplus \langle \bar{a}^*_7 \rangle$.
It follows that for the isotopic subgroup $\mathrm{V}=\langle \bar{d}_1^*  + 0\rangle$ we have
\begin{equation*}
 \Big( D(\mathcal{W}^{\mathrm{root}}), q_{\mathcal{W}^{\mathrm{root}}} \Big)\Big|_{\mathrm{V}^\perp/\mathrm{V}} = 
\left( \mathbb{Z}/8, \left[\frac{1}{8} \right] \right) \;.
\end{equation*}
$\Psi^{-1}(\mathrm{V})$ is an even overlattice of $\mathcal{W}^{\mathrm{root}}$ with $\Psi^{-1}(\mathrm{V})/\mathcal{W}^{\mathrm{root}}\cong \mathbb{Z}_2$. 
For $\mathrm{L}=\mathrm{D}_8$ the overlattice is $\mathrm{D}_8^+ = \Psi^{-1}(\bar{d}_1^*)= \mathrm{E}_8$, and for $\mathrm{L}=\mathrm{A}_7$
the overlattice is $\Psi^{-1}(0)=\mathrm{A}_7$. Hence, we have $\mathcal{W} = \Psi^{-1}(\mathrm{V})=\mathrm{E}_8 \oplus \mathrm{A}_7$.
It follows that $\mathrm{NS}(\mathbf{X}) \cong \langle 8 \rangle  \oplus \mathrm{E}_8  \oplus \mathrm{E}_8$.
By \cite[Thm. 7.1]{Shimada}, the transcendental lattice $\mathrm{T}_{\mathbf{X}}$ satisfies
\begin{equation*}
 \Big( D(\mathrm{T}_{\mathbf{X}}), q_{\mathrm{T}_{\mathbf{X}}} \Big) \cong \Big( D(\mathcal{W}), -q_{\mathcal{W}} \Big) \;.
\end{equation*}
Thus, the transcendental lattice has signature $(2,3)$ and discriminant group $\mathbb{Z}/8$.
By \cite[Corollary 2.9(iii)]{Morrison84}, a lattice $\mathrm{T}$ of signature $(2,3)$ occurs as transcendental 
lattice of a K3 surface if and only if $\mathrm{T} \cong \mathrm{H}^2 \oplus \mathrm{T}'$ where $\mathrm{T}'$ is a negative definite lattice of rank $1$.
Thus, we have $\mathrm{T}_{\mathbf{X}} =  \mathrm{H}^2 \oplus \langle -8\rangle$.

\end{document}